\documentclass[11pt,a4paper]{article}
\usepackage{authblk}
\usepackage[french,english]{babel}
\usepackage[utf8]{inputenc}
\usepackage{amsmath}
\usepackage{amsfonts}
\usepackage{amssymb}
\usepackage{hyperref}
\hypersetup{
    colorlinks=true,
    urlcolor=blue,
    linkcolor=blue,
    breaklinks=true,
    citecolor=blue,
}
\usepackage{url}
\usepackage{xcolor,color}
\colorlet{darkgreen}{green!50!black}
\usepackage{amsbsy}
\usepackage{amsmath}
\usepackage{dsfont}
\usepackage{arydshln}

\usepackage{geometry}
\usepackage{graphicx}
\usepackage{tikz}
\usepackage[numbers]{natbib}
\usepackage[margin=1cm]{caption}

\usepackage{amsthm}
\usepackage{oubraces}
\usepackage[shortlabels]{enumitem}
\usepackage{bm}

\theoremstyle{theorem}
\newtheorem{theorem}{Theorem}

\newtheorem{lemma}[theorem]{Lemma}
\newtheorem{proposition}[theorem]{Proposition}

\theoremstyle{definition}
\newtheorem{definition}[theorem]{Definition}

\newtheorem{example}[theorem]{Example}

\usepackage[font={small,it}]{caption}
\usetikzlibrary{patterns}

\newlength{\hatchspread}
\newlength{\hatchthickness}
\newlength{\hatchshift}
\newcommand{\hatchcolor}{}
\tikzset{hatchspread/.code={\setlength{\hatchspread}{#1}},
         hatchthickness/.code={\setlength{\hatchthickness}{#1}},
         hatchshift/.code={\setlength{\hatchshift}{#1}},
         hatchcolor/.code={\renewcommand{\hatchcolor}{#1}}}
\tikzset{hatchspread=3pt,
         hatchthickness=0.4pt,
         hatchshift=0pt,
         hatchcolor=black}
\pgfdeclarepatternformonly[\hatchspread,\hatchthickness,\hatchshift,\hatchcolor]
   {custom north west lines}
   {\pgfqpoint{\dimexpr-2\hatchthickness}{\dimexpr-2\hatchthickness}}
   {\pgfqpoint{\dimexpr\hatchspread+2\hatchthickness}{\dimexpr\hatchspread+2\hatchthickness}}
   {\pgfqpoint{\dimexpr\hatchspread}{\dimexpr\hatchspread}}
   {
    \pgfsetlinewidth{\hatchthickness}
    \pgfpathmoveto{\pgfqpoint{0pt}{\dimexpr\hatchspread+\hatchshift}}
    \pgfpathlineto{\pgfqpoint{\dimexpr\hatchspread+0.15pt+\hatchshift}{-0.15pt}}
    \ifdim \hatchshift > 0pt
      \pgfpathmoveto{\pgfqpoint{0pt}{\hatchshift}}
      \pgfpathlineto{\pgfqpoint{\dimexpr0.15pt+\hatchshift}{-0.15pt}}
    \fi
    \pgfsetstrokecolor{\hatchcolor}
    \pgfusepath{stroke}
   }

\newcommand{\nb}[1]{#1}
\newcommand{\nr}[1]{}

\author[$1$]{Jules Flin}
\author[$1$]{Sandro Franceschi}

\affil[$1$]{T\'el\'ecom SudParis, Institut Polytechnique de Paris}

\title{Reflected Brownian Motion in a wedge: 
\\ sum-of-exponential absorption probability at the vertex
\\ and differential properties}
\date{}

\begin{document}

\maketitle
\thispagestyle{empty}

\abstract{We study a Brownian motion with drift in a wedge of angle $\beta$ which is obliquely reflected on each edge along angles $\varepsilon$ and $\delta$. We assume that the classical parameter $\alpha=\frac{\delta+\varepsilon-\pi}{\beta}$ is greater than~$1$ and we focus on transient cases where the process can either be absorbed at the vertex or escape to infinity. We show that $\alpha\in\mathbb{N}$ is a necessary and sufficient condition for the absorption probability, seen as a function of the starting point, to be written as a finite sum of terms of exponential product form. In such cases, we give expressions for the absorption probability and its Laplace transform. When $\alpha\in \mathbb{Z}+\frac{\pi}{\beta}\mathbb{Z}$ we find an explicit \nb{differentially-}algebraic expression for the Laplace transform. Our results rely on Tutte's invariant method and a recursive compensation approach.}

\section{Introduction}

\paragraph{Context}

In dimension one, it is known that a standard Brownian motion with positive drift $\mu >0$ started at $u>0$ has probability $e^{-2  \mu u}$ to reach $0$. A simple way of achieving this result is to use Girsanov's theorem and the reflection principle.
In dimension 2, we consider an obliquely reflected Brownian motion in a cone with drift belonging to the interior of the cone and directions of reflection strongly oriented towards the apex of the cone. A phenomenon of competition between the reflections and the drift appears and the process is either absorbed at the vertex or escapes to infinity. Lakner, Liu, and Reed~\cite{Lak-Liu-Reed-2023} studied this absorption phenomenon and showed the existence and uniqueness of a solution to the absorbed process.
Ernst et al.~\cite{ErFrHu-20} were able to obtain a general formula for the probability of absorption at the vertex using Carleman's boundary value problems theory. In particular, they characterised the cases where this probability has an exponential product form, \textit{i.e.} when the reflection vectors are opposite. Franceschi and Raschel~\cite{franceschi_dual_2022} then generalised this result to higher dimensions by showing that the coplanarity of the reflection vectors was a necessary and sufficient condition for the absorption probability to have an exponential product form. In a sense, this condition can be seen as dual to the classical skew symmetry condition discovered by Harrison and Williams \cite{harrison_diffusion_1978,harrison_multidimensional_1987,Wi-87} which characterises cases where the stationary distribution is exponential.
In dimension 2, when the process is recurrent, Dieker and Moriarty~\cite{dieker_moriarty_2009}, preceded by Foshini~\cite{Foschini} in the symmetric case, determined a necessary and sufficient condition for the stationary distribution to be a sum of exponentials terms of product form. It is therefore very natural to look for an analogous result to the one of Dieker and Moriarty. This article aims to find, when the process is transient, a necessary and sufficient condition for the absorption probability to be a sum-of-exponentials function of the starting point and to compute this probability. We also identify other remarkable cases where the Laplace transform of the absorption probability is \nb{differentially-algebraic (D-algebraic), \textit{i.e.} solution of a polynomial equation in the function, its derivatives, and the independent variables, with coefficients in $\mathbb R$}.


\paragraph{Key parameter and main results}

To present our results in more detail, we need to introduce a few parameters usually used to define a semimartingale reflecting Brownian motion (SRBM). We define the cone 
$
C:=\{(r\cos(t),r\sin(t)) : r \geqslant 0 \text{ and } 0\leqslant t \leqslant \beta \}
\label{eq:C}
$
of angle $\beta\in (0,\pi)$ and consider $\widetilde Z_t$ an obliquely reflected standard Brownian motion with drift $\widetilde \mu\in\mathbb{R}^2$ of angle $\theta\in (-\pi,\pi]$ and reflection vectors of angles $\delta\in (0,\pi)$ and $\varepsilon\in (0,\pi)$, see Figure~\ref{fig:angles} to visualize these angles. We define
\begin{equation}
\alpha:=\frac{\delta+\varepsilon-\pi}{\beta}
\label{eq:alpha}
\end{equation} 
which is a famous key parameter in the SRBM literature. As a general rule, such a process is most of the time studied in the literature in the case where $\alpha< 1$, \textit{i.e.} in the case where the process is a semimartingale markov process, see the seminal work of Varadhan and Williams \cite{varadhan_brownian_1985,Wi-85b}. We will not give here a precise mathematical definition of the process, 
which can be found in many articles, see the survey of Williams \cite{williams_semimartingale_1995}. 
We will simply point out that it behaves like a standard Brownian motion with drift inside the cone, it is reflected in a given direction when it touches an edge (being pushed by the local time on the boundary) and it spends zero time at the vertex of the cone. The famous skew symmetric condition, where the stationary distribution has an exponential product form, corresponds to $\alpha=0$, and Dieker and Moriarty's condition for a sum-of-exponential stationary density corresponds to $\alpha \in -\mathbb{N}\cup\{0\}$. The dual skew symmetric case \cite{ErFrHu-20,franceschi_dual_2022}, where the escape probability has an exponential product form, correspond to $\alpha=1$.
For our purposes, in this article, we will assume that
\begin{equation}
\alpha\geqslant 1 
\label{eq:alphageq1}
\end{equation}  
so that the process can be trapped at the vertex and we will consider transient cases where the drift $\widetilde \mu$ belongs to the interior of the cone $C$, that is when $\theta\in (0,\beta)$. We define the first hitting time of the vertex
\begin{equation*}
T:=\inf \{ t>0 : \widetilde Z_t =0 \}.
\label{eq:T}
\end{equation*} 
The article~\cite{Lak-Liu-Reed-2023} makes a detailed study of the absorbed process, its existence, and its uniqueness in this case. As explained in the articles \cite{ErFrHu-20,franceschi_dual_2022,Lak-Liu-Reed-2023}, by following the results from Taylor and Williams \cite{TaWi-93}, when $\alpha\geqslant 1$ the process $Z_t$ is well defined until it hits the vertex at time $T$, which amounts to considering the process $(\widetilde Z_t)_{0 \leqslant t \leqslant T}$.

\begin{figure}[h]
\centering
\includegraphics[width=0.30\linewidth]{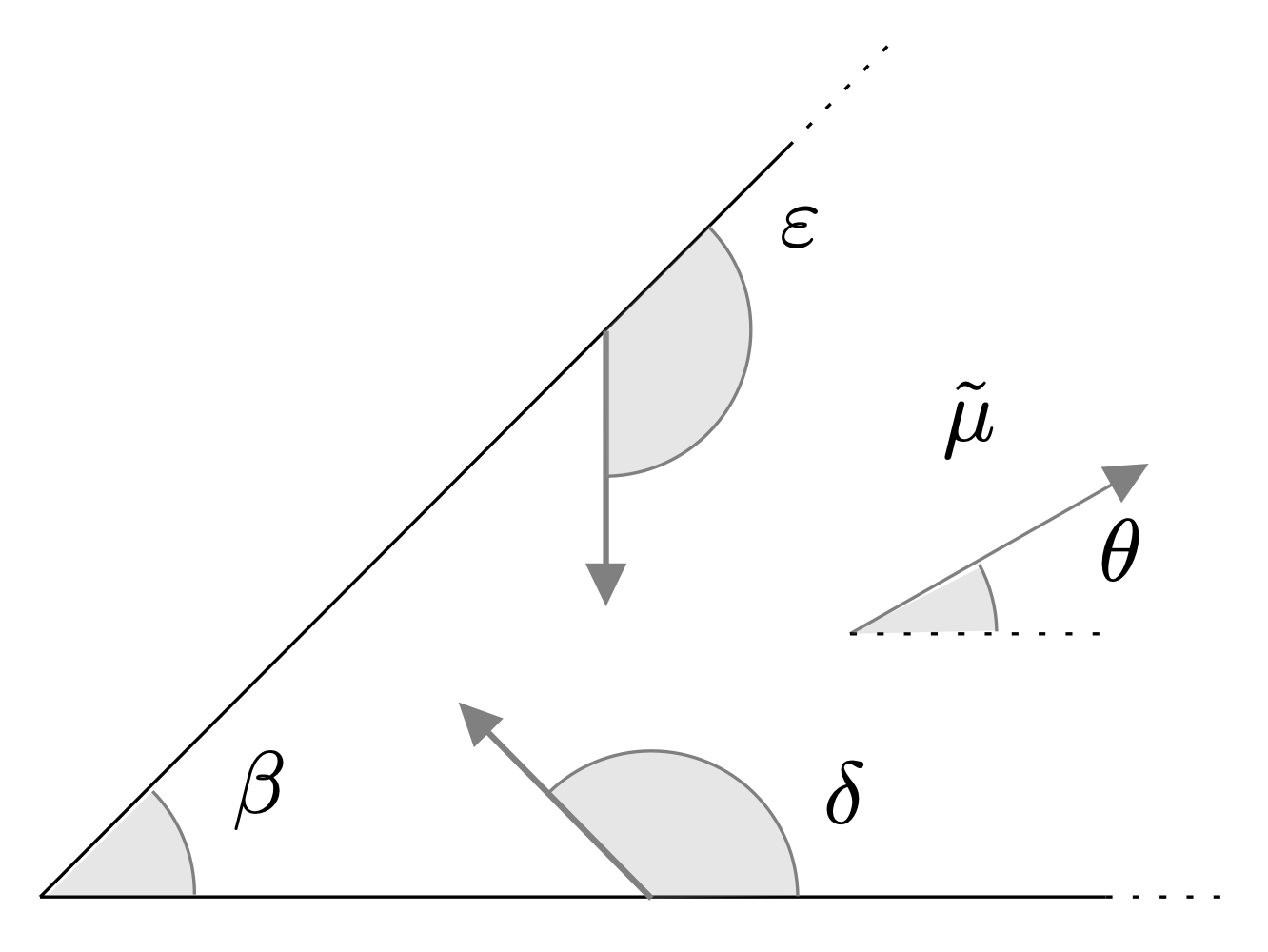}\\
\caption{Configuration of the angles used to describe the model.}
\label{fig:angles}
\end{figure}

The main results of the article are as follows. We prove that the absorption probability at the vertex $\mathbb{P}(T<\infty)$ is a sum-of-exponential function of the starting point if and only if 
\begin{equation}
\alpha\in \nb{\mathbb{N}:=\{1,2,3,\dots\}}
\label{eq:alphainN}
\end{equation}
plus the\nr{$\;\,$small technical} condition
\begin{equation}
\forall j\in\{1,\dots,2\alpha-2\}, \ \theta-2\delta+j\beta\not\equiv 0\text{ mod}(\pi)
\label{eq:polesimpleintro}
\end{equation}
which excludes cases where there are multiple poles in the Laplace transform. 
In fact, our results are much more accurate than that. Assuming that~\eqref{eq:alphainN} and~\eqref{eq:polesimpleintro} hold, if $(u,v)$ is the starting point of the process (mapped onto the quadrant, see \eqref{eq:uvtilde}) the absorption probability is of the form
\begin{equation}\label{eq:soe}
\mathbb{P}_{(u,v)}(T<\infty)=\sum_{k=1}^{2\alpha-1} c_k \exp\left(a_ku+b_kv\right),
\end{equation}

\noindent where the coefficients $a_k$, $b_k$ and $c_k$ are computed explicitly in Theorem~\ref{thm:explicitproba}.
In the cases where $\theta-2\delta+j\beta\equiv 0\text{ mod}(\pi)$ for some $j\in\{1,\dots,2\alpha-2\}$,  the absorption probability has the form
$
\mathbb P_{(u,v)}(T<\infty)=\sum_{k=1}^{2\alpha-1}A_k(u,v)\exp\left(a_ku+b_kv\right),
$
where the $A_k$ are affine functions of the variables $u$ and $v$, see last paragraph of the article.

In Theorem~\ref{thm:laplace2} we state another more general and stronger result which explicitly determines the Laplace transform of the absorption probability in terms of a \nr{conformal gluing}\nb{Gauss hypergeometric} function when 
$$
\alpha\in\mathbb{Z}+\frac{\pi}{\beta}\mathbb{Z}.
$$
In this case, we also find the differential nature of the Laplace transform.
In other words we find sufficient conditions on $\alpha$ for the Laplace transform to be rational, algebraic (\textit{i.e.} satisfying a polynomial equation with coefficients in the
field of rational functions\nb{$\,\;$over $\mathbb R$}), D-finite (\textit{i.e.} satisfying a linear differential equation with coefficients in the field of rational functions\nb{$\,\;$over $\mathbb R$}) or D-algebraic (\textit{i.e.} satisfying a polynomial differential equation with coefficients in the field of rational functions\nb{$\,\;$over $\mathbb R$}).
The differential nature of the Laplace transform reflects in various ways on the absorption probability itself. For example, if it is rational it implies that the absorption probability is a linear combination of exponentials multiplied by polynomials. If it is D-algebraic it will give a recurrence relation for the moments. We refer to the introduction of~\cite{BoMe-El-Fr-Ha-Ra} which explains in more detail the interest of such a classification in this hierarchy of functions:
\begin{equation}
\text{rational}\subset\text{algebraic}\subset\text{D-finite}\subset\text{D-algebraic}.
\label{eq:hierarchy}
\end{equation}
The following table gives sufficient conditions for the Laplace transform to belong to this hierarchy.
\begin{center}
\begin{tabular}{|c|c|c|c|}
\hline
rational & algebraic & D-finite & D-algebraic
\\
\hline
$\alpha\in\mathbb{N}$ & $\pi/\beta\in\mathbb{Q}$ and $\alpha\in\mathbb Z+\frac{\pi}{\beta}\mathbb Z$ & $\alpha\in\mathbb N+\frac{\pi}{\beta}\mathbb Z$ & $\alpha\in\mathbb Z+\frac{\pi}{\beta}\mathbb Z$
\\
\hline
\end{tabular}
\end{center}

\paragraph{Plan and strategy of proof}

Section~\ref{sec:prelim} presents the preliminaries needed to prove our results. For technical reasons, we first transfer the problem initially defined in a wedge into a quadrant thanks to a simple linear transform. The starting point of the proof is a kernel functional equation satisfied by the Laplace transform of the absorption probability as a function of the starting point. This equation is derived from a partial differential equation solved by the probability of absorption. This functional equation leads to a boundary value problem (BVP) already studied in \cite{ErFrHu-20}. In Section~\ref{sec:tutte}, we apply successfully Tutte's invariant method~\cite{tutte_chromatic_1995} to this BVP finding some \textit{decoupling functions}, in a similar way to what was done in the recurrent case for the stationary distribution~\cite{BoMe-El-Fr-Ha-Ra,franceschi_tuttes_2016}. We then compute explicitly the Laplace transform, see Theorems~\ref{thm:laplace} and~\ref{thm:laplace2}. Inverting the bivariate Laplace Transform is no easy task because of a complicated factorization of a two variable polynomial by the kernel. In Section~\ref{sec:compensation}, we then offer a geometrical way to construct the solutions inspired by the compensation approach developed with success in the discrete case for some queueing problems and random walks by Adan, Wessels, and Zijm \cite{adan_wessels_zijm_compensation_93}. 

\paragraph{Related literature and perspectives}

This paper develops an original way of showing these results, which is an alternative, although closely related, to the Dieker and Moriarty \cite{dieker_moriarty_2009} method in the recurrent case. Another approach to show our results might have been to use an equivalence based on time reversal and developed very recently by Harrison \cite{Ha-22} which shows that the hitting time of the vertex is inherently connected to the stationary distribution of a certain dual process, and then apply the results of \cite{dieker_moriarty_2009} to a certain trapezoid described in \cite{Ha-22}.

It is also important to mention the strong links between the results of this article and the Weil chambers and reflection groups. For example, Biane, Bougerol, and O’Connell~\cite{biane_bougerol_oconnell_2005} express the persistence probability, that is the probability that a Brownian motion with drift stays forever in a Weyl chamber, as a sum-of-exponential. We may also mention the article by Defosseux~\cite{defosseux_2016} which expresses similar results for a space-time Brownian motion. 

It is also possible to interpret our problem as the study of the probability of triple collisions for transient competing particle systems with asymmetric collisions. Indeed, a Brownian motion reflected in a quadrant is nothing more than the gap process of such a system made of three particles, and reaching the vertex of the quadrant is equivalent to a triple collision. A very interesting literature is devoted to the study of the absence or presence of such collisions, and as we cannot claim exhaustiveness in these few lines we will limit ourselves to mentioning the articles by Ichiba, Karatzas, and Shkolnikov~\cite{Ich-Kar-Shkol-2013}, Bruggeman and Sarantsev~\cite{bruggeman_sarantsev_2018}\nb{, and Sarantsev~\cite{sarantsev}}. \nb{In stochastic finance, Banner, Fernholz and Karatzas~\cite{bfk} shown strong connections between rank-based models (such as Atlas models) and the reflected Brownian motion.} 

\nb{In queueing theory, the reflected Brownian motion 
can be thought as the (scaled) limit of the queueing length process. Such convergence results are refered to as \textit{heavy traffic limit theorems}, see the founding article by  Harrison~\cite{harrison_diffusion_1978} and the classic book by Whitt~\cite{whitt}. This is still a very active research fields with many applications, see for example~\cite{atakumar,boxma}.
}

To conclude this introduction, we must emphasise that this article is an important step towards a more ambitious outcome. Indeed, we believe that the present results can be extended. More precisely, we believe, as was done in the recurrent case for the stationary distribution in the article of~Bousquet-M{\'e}lou et~al. \cite{BoMe-El-Fr-Ha-Ra}, that it is possible to characterise the algebraic and differential nature of the Laplace transforms of the absorption probability. In a sense, this would 
exhaustively rank the complexity of the absorption probability in the hierarchy~\eqref{eq:hierarchy} according to the value of $\alpha$. Such a result, which would provide sufficient but also necessary conditions, would require difference Galois theory which is beyond the scope of this article.

\section{Preliminaries}
\label{sec:prelim}

\paragraph{From the cone to the quadrant} 
\nb{The results presented in this paper are particularly neat when expressed in terms of the angles that define the process in a cone, whereas the proofs are simpler for a process in the quadrant. This is why} the results stated in the introduction for the standard Brownian motion $\widetilde Z_t$ reflected in a cone $C$ of angle $\beta\in(0,\pi)$, drift angle $\theta$ and reflection angles $\delta$ and $\varepsilon$ will be proved by considering a Brownian motion $Z_t$ reflected in the quadrant $\mathbb{R}_+^2$ with a \nb{positive-definite} covariance matrix, a drift and a reflection matrix noted respectively
$$\Sigma=\left(\begin{array}{cc}
   \sigma_{11}  & \sigma_{12}  \\
   \sigma_{12}  & \sigma_{22}
\end{array}\right),\quad\quad \mu=\left(\begin{array}{c}
     \mu_1 \\
     \mu_2 
\end{array}\right),\quad\quad R=\left(R^1,R^2\right)=\left(\begin{array}{cc}
    1 & -r_2 \\
    -r_1 & 1
\end{array}\right) ,$$
which satisfy the following relations
\begin{equation}
   \cos \beta=\frac{-\sigma_{12}}{\sqrt{\sigma_{11}\sigma_{22}}} ,
   \quad
   \tan \theta =\frac{\mu_2\sqrt{\det(\Sigma)}}{\sigma_{22}\mu_1-\sigma_{12}\mu_2} ,
   \quad
   \tan \delta = \frac{-\sqrt{\det(\Sigma)}}{r_2\sigma_{22}+\sigma_{12}} ,
   \quad
   \tan \varepsilon = \frac{-\sqrt{\det(\Sigma)}}{\sigma_{11}r_1+\sigma_{12}}.
   \label{eq:betathetadeltaepsilon}
\end{equation}
The study of these two processes is equivalent by considering $\phi$ a simple bijective linear transform defined by 
\begin{equation*}
    \phi:=\left(\begin{array}{cc}
    \sqrt{\sigma_{11}} & 0 \\
    0 & \sqrt{\sigma_{22}}
\end{array}\right)\left(\begin{array}{cc}
\sin(\beta) & - \cos(\beta)\\
0 & 1 
\end{array}\right)\in \text{GL}(2,\mathbb R)
\end{equation*}
which maps the cone $C$ onto the first quadrant $\mathbb{R}_+^2=\phi(C)$, and $\widetilde Z_t$ onto $Z_t=\phi(\widetilde{Z}_t)$. We have of course $\phi(0)=0$ and $T=\inf \{ t>0 : \widetilde Z_t =0 \}=\inf \{ t>0 : Z_t =0 \} $.
It is then equivalent to compute the absorption probability for the process $\widetilde Z_t$ starting from $(\widetilde u, \widetilde v)\in C$ and for the process $Z_t$ starting from $(u,v)=\phi(\widetilde u, \widetilde v) \in\mathbb{R}_+^2$. We denote the escape probability
\begin{equation}
f(u,v):= \mathbb{P}_{(u,v)}(T<\infty) 
\quad\text{and}\quad
\widetilde f(\widetilde u,\widetilde v):=f(\phi(\widetilde u, \widetilde v)).
\label{eq:uvtilde}
\end{equation}
This linear transform doesn't affect the form of the absorption probability. More precisely, the absorption probability $f(u,v)$ is a sum-of-exponential, given by~\eqref{eq:soe}, if and only if $\widetilde f(\widetilde u,\widetilde v)$ is a sum-of-exponential, given by
\begin{equation*}
\widetilde f(\widetilde u,\widetilde v)=\sum_{k=1}^{2\alpha-1}c_k \exp\left(a_k\sqrt{\displaystyle\frac{\sigma_{22}}{\det(\Sigma)}}\widetilde u+\frac{1}{\sqrt{\sigma_{22}}}\left[b_k-\frac{a_k\sigma_{12}}{\sqrt{\det(\Sigma)}}\right]\widetilde v\right).
\end{equation*}
\nb{Note that thanks to this linear transformation, our results generalise immediately to all Brownian motions with any covariance matrix (possibly different from the identity) in any convex cone.}


\begin{figure}[t]
\begin{center}
\includegraphics[width=0.7\linewidth]{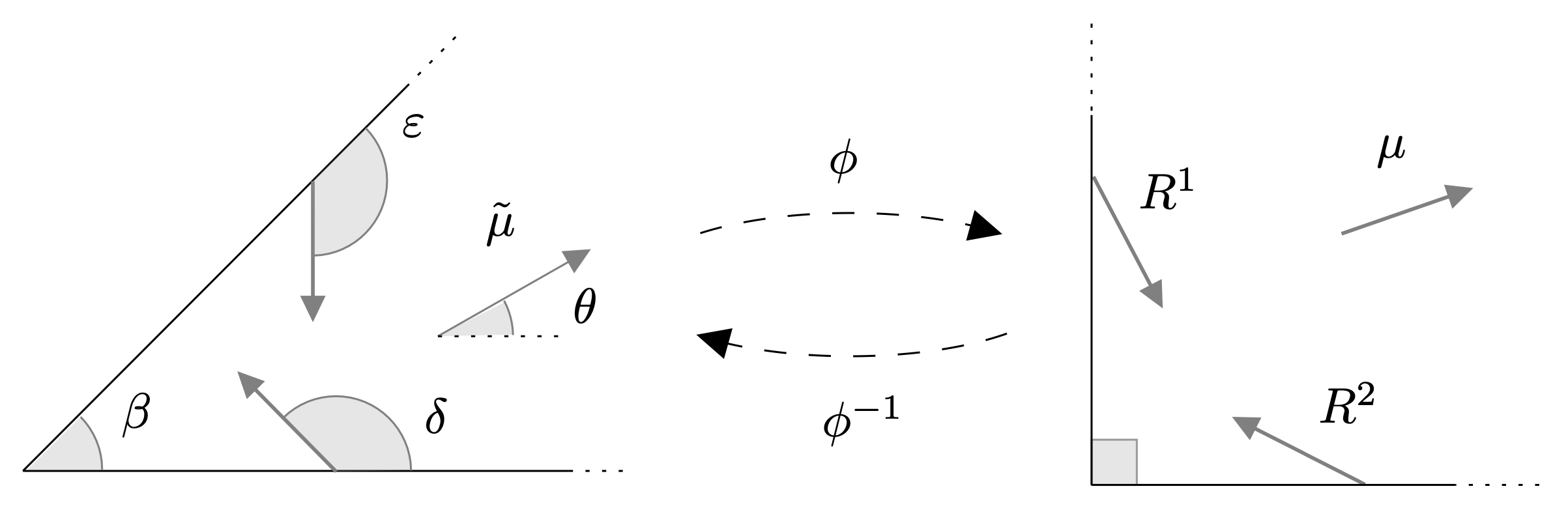}
\caption{Equivalence between the reflected Brownian motion $\widetilde Z_t$ in a $\beta$-wedge with drift $\Tilde{\mu}$ and reflection angles $\delta$ and $\varepsilon$ and the reflected Brownian motion $Z_t$ in the first quadrant with drift $\mu$ and reflection vectors $R^1$ and $R^2$.}
\label{fig:model}
\end{center}
\end{figure}


\paragraph{Partial differential equation} 
The escape probability of the process $Z_t$ starting from $(u,v)$ defined by
$$g(u,v):=1-f(u,v)=\mathbb P_{(u,v)}(T=\infty)$$
satisfies the following partial differential equation, see Proposition 11 in \cite{ErFrHu-20}.
The function $g$ is both bounded and continuous in the quarter plane and on its boundary and continuously differentiable in the quarter plane and on its boundary (except perhaps at the corner), and satisfies the elliptic partial differential equation 
    \begin{equation}\label{eq:PDE}
    \mathcal{G}g:=\left(\frac{1}{2}\nabla \cdot \Sigma\nabla +\mu \cdot\nabla\right)g=0\text{ on }\mathbb R_+^2
    \end{equation}
     with oblique Neumann boundary conditions 
     \begin{equation}\label{eq:neumann}
     \partial_{R^1}g(0,\,\cdot\,):=\left(R^1\cdot \nabla\right) g(0,\,\cdot\,)=0\text{,  }\partial_{R^2}g(\,\cdot\,,0):=\left(R^2\cdot \nabla\right)g(\,\cdot\,,0)=0\text{ on }\mathbb R_+
     \end{equation}
      and the limit conditions 
\begin{equation}\label{eq:limits}
g(0,0)=0,\quad\quad \lim_{\|(u,v)\|\to\infty}g(u,v)=1.
\end{equation}
The absorption probability $f=1-g$ satisfies the same partial differential equation replacing~\eqref{eq:limits} with the appropriate limit conditions. 

\paragraph{Laplace transform and functional equation}
We define the Laplace transform of the escape probability $g(u,v)$ by
$$\varphi(x,y):=\int\!\!\!\!\!\int_{\mathbb R_+^2}\mathbb P_{(u,v)}(T=\infty)e^{-xu-yv}\text{d}u\text{d}v$$
and the Laplace transforms of the escape probabilities $g(0,v)$ and $g(u,0)$ when the process starts from the boundaries by
\begin{equation}
\varphi_1(y):=\int_{\mathbb R_+}\mathbb P_{(0,v)}(T=\infty)e^{-yv}\text{d}v,\quad\quad \varphi_2(x):=\int_{\mathbb R_+}\mathbb P_{(u,0)}(T=\infty)e^{-xu}\text{d}u.
\label{eq:phi1phi2}
\end{equation}
One can easily use integrations by parts to translate the partial differential equation made of the three conditions \eqref{eq:PDE}, \eqref{eq:neumann}, \eqref{eq:limits} into a functional equation for the Laplace transforms.

\begin{proposition}[{Prop. 12 in \cite{ErFrHu-20}}]
  The Laplace transforms  $\varphi$, $\varphi_1$ and $\varphi_2$ satisfy the following kernel functional equation, for $(x,y)\in\mathbb{C}^2$ such that $\Re x >0$ and $\Re y >0$ we have
    \begin{equation}\label{eq:functionaleq}
         K(x,y)\varphi(x,y)=k_1(x,y)\varphi_1(y)+k_2(x,y)\varphi_2(x)
    \end{equation}
    where
    \begin{equation}\label{eq:kernel}
    \left\{\begin{array}{l}
         \displaystyle K(x,y):=\frac{1}{2}(x,y)^\intercal \cdot \Sigma (x,y)^\intercal +\mu \cdot (x,y)^\intercal = \frac{1}{2}(\sigma_{11}x^2+\sigma_{22}y^2+2\sigma_{12}xy)+\mu_1x+\mu_2y , \\[0.2cm]
         \displaystyle k_1(x,y):=\frac{\sigma_{11}}{2}(x+r_1y)+\sigma_{12}y+\mu_1 ,\\[0.2cm]
         \displaystyle k_2(x,y):=\frac{\sigma_{22}}{2}(r_2x+y)+\sigma_{12}x+\mu_2 .
    \end{array}\right.
    \end{equation}
\end{proposition}

\paragraph{Study of the kernel $K$ and uniformization}

To solve functional equation (\ref{eq:functionaleq}), we first need to study $K$, and more precisely its vanishing set
\begin{equation}\label{eq:S}
\mathcal{S}:=\{(x,y)\in \mathbb C^2 : K(x,y)=0\} . 
\end{equation}
For $x\in \mathbb C$, the equation $K(x,y)=0$ in $y$ is quadratic, and has therefore two solutions $Y^+(x)$ and $Y^-(x)$ in $\mathbb C$:
\begin{equation}
    Y^\pm(x):= \frac{-\sigma_{12}x-\mu_2\pm \sqrt{(\sigma_{12}x+\mu_2)^2-\sigma_{22}(\sigma_{11}x^2+2\mu_1 x)}}{\sigma_{22}}.
    \label{eq:Ypm}
\end{equation}
Likewise, we define $X^+(y)$ and $X^-(y)$ to be the two solutions of the equation $K(x,y)=0$ in the variable $x$.
The curve $\mathcal{S}$ can be thought of as the image of the multivalued function $Y$ (\textit{resp}. $X$) which has two ramification points $x^+$ and $x^-$ (\textit{resp.} $y^+$ and $y^-$) given by
\begin{equation}
x^\pm:=\frac{(\mu_2\sigma_{12}-\mu_1\sigma_{22})\pm \sqrt{(\mu_2\sigma_{12}-\mu_1\sigma_{22})^2+\det(\Sigma)\mu_2^2}}{\det(\Sigma)}.
\label{eq:xpm}
\end{equation}
The branches $X^+$ and $X^-$ are analytic on $\mathbb C\setminus((-\infty,y^-]\cup[y^+,+\infty))$, and $Y^+$ and $Y^-$ are analytic on $\mathbb C\setminus ((-\infty,x^-]\cup[x^+,+\infty))$. 

\nb{\begin{lemma}\label{lemma:irreducible}
The kernel $K$ is irreducible over $\mathbb C[X,Y]$.
\end{lemma}
\begin{proof}
Suppose $K=AB$ for two non-constant polynomials $A,B\in\mathbb C[X,Y]$. Given the degree of $K$, the polynomials $A$ and $B$ must be of degree $1$:
\begin{equation}
        A(x,y):=a_1x+a_2y+a_3,\quad\quad B(x,y):=b_1x+b_2y+b_3
\end{equation}
where $a_2 b_2=\sigma_{22}/2 \neq 0$.
Solving equation $K(x,y)=A(x,y)B(x,y)=0$ in $y$ yields 
$$\{Y^+(x),Y^-(x)  \}=\{-( a_1x+a_3)/a_2 ,-(b_1x+b_3)/b_2  \}.$$ This means that $Y^+$ and $Y^-$ are affine functions which implies that the polynomial under the square root of Equation~\eqref{eq:Ypm} is the square of an affine function. This is equivalent to say that 
$$(\sigma_{12}x+\mu_2)^2-\sigma_{22}(\sigma_{11}x^2+2\mu_1 x)$$ 
has a double root and then  to the fact that the discriminant 
$$(\mu_2\sigma_{12}-\mu_1\sigma_{22})^2+\det(\Sigma)\mu_2^2$$ 
(already computed in~\eqref{eq:xpm}) is equal to $0$. This is not possible because $\det \Sigma >0$ since $\Sigma$ is a positive-definite covariance matrix. Hence there is no such decomposition of $K$, and $K$ is therefore irreducible.
\\
We can sketch an alternative proof. According to the classification of conics (see~\cite{Catalog}, p.63), the nature of the conic $\{(x,y)\in\mathbb{R}^2 : K(x,y)=0 \}$ only depends on the signs of $\det(\Sigma)$ and the following block matrix determinant
    \begin{equation}
        \Delta:=\left|\begin{array}{c;{2pt/2pt}c}
        \Sigma & \mu \\ \hdashline[2pt/2pt]
        \mu^\intercal & 0 
    \end{array}\right|=-(\mu_2 \sigma_{11}^2 +\mu_1 \sigma_{22}^2-2\mu_1\mu_2\sigma_{12})=-\Big(\mu^\intercal \Sigma^{-1}\mu \Big)\det(\Sigma) .
    \end{equation}
    The matrix $\Sigma$ is positive-definite, and so is its inverse $\Sigma^{-1}$, hence $\det(\Sigma)>0$ and $\Delta<0$, which corresponds to a non-degenerate 
ellipse. If $K=AB$ for two polynomials of degree $1$, it would imply that the ellipse would be equal to the union of two lines and therefore would be degenerate. We deduce again that $K$ is irreducible. 
\\
Note that the discriminant of the first proof is linked to the determinant of the second proof since $(\mu_2\sigma_{12}-\mu_1\sigma_{22})^2+\det(\Sigma)\mu_2^2=-\sigma_{22} \Delta$.
\end{proof}}

It will be handy to work with the following rational uniformization of $\mathcal{S}$, first stated in \cite[Proposition~5]{franceschi_kurkova}, defined by
\begin{equation}
\left(\boldsymbol{\mathrm{x}},\boldsymbol{\mathrm{y}}\right)(s):=\left(\frac{x^++x^-}{2}+\frac{x^+-x^-}{4}\left(s+\frac{1}{s}\right),\frac{y^++y^-}{2}+\frac{y^+-y^-}{4}\left(\frac{s}{e^{i\beta}}+\frac{e^{i\beta}}{s}\right)\right),
\label{eq:unif}
\end{equation}
which is such that
\begin{equation*}
\mathcal{S}=\{(\boldsymbol{\mathrm{x}}(s),\boldsymbol{\mathrm{y}}(s)),\;s\in\mathbb{C}^*\}.
\end{equation*}
In the following, we adopt the notation
\begin{equation}\label{eq:q}
    \boldsymbol{\mathrm{q}}:=e^{2i\beta}.
\end{equation}
The functions $\boldsymbol{\mathrm{x}}$ and $\boldsymbol{\mathrm{y}}$ satisfy the following invariance properties: for all $s\in\mathbb C^*$
\begin{equation}\label{eq:invariance}
\boldsymbol{\mathrm{x}}(s)=\boldsymbol{\mathrm{x}}(s^{-1})\text{ and }\boldsymbol{\mathrm{y}}(s)=\boldsymbol{\mathrm{y}}(\boldsymbol{\mathrm{q}}s^{-1}).
\end{equation}


\begin{lemma} There exists $C_1,C_2\in\mathbb C$ such that for all $s\in\mathbb C^*$ the polynomials defined in~\eqref{eq:kernel} satisfy
    \begin{equation*}
     k_1(\boldsymbol{\mathrm{x}}(s),\boldsymbol{\mathrm{y}}(s))=C_1 \frac{(s-s_0')(s-s_1)}{s},\quad\quad  k_2(\boldsymbol{\mathrm{x}}(s),\boldsymbol{\mathrm{y}}(s))=C_2 \frac{(s-s_0'')(s-s_2)}{s},
\end{equation*}
with
\begin{equation}\label{eq:si}
     s_0':=e^{i(2\beta-\theta)},\quad\quad s_0'':=e^{-i\theta},\quad\quad s_1:=e^{i(\theta+2\varepsilon)}, \quad\quad s_2:=e^{i(\theta-2\delta)}.
\end{equation}
\label{lemma:k(s)}
\end{lemma}

\begin{proof}
For $i\in\{1,2\}$, $k_i(\boldsymbol{\mathrm{x}}(s),\boldsymbol{\mathrm{y}}(s))=0$ is a degree-two polynomial equation whose roots can be computed with some basic trigonometry using~\eqref{eq:betathetadeltaepsilon}.
\end{proof}

\begin{figure}[t]
\begin{center}
\includegraphics[width=\linewidth]{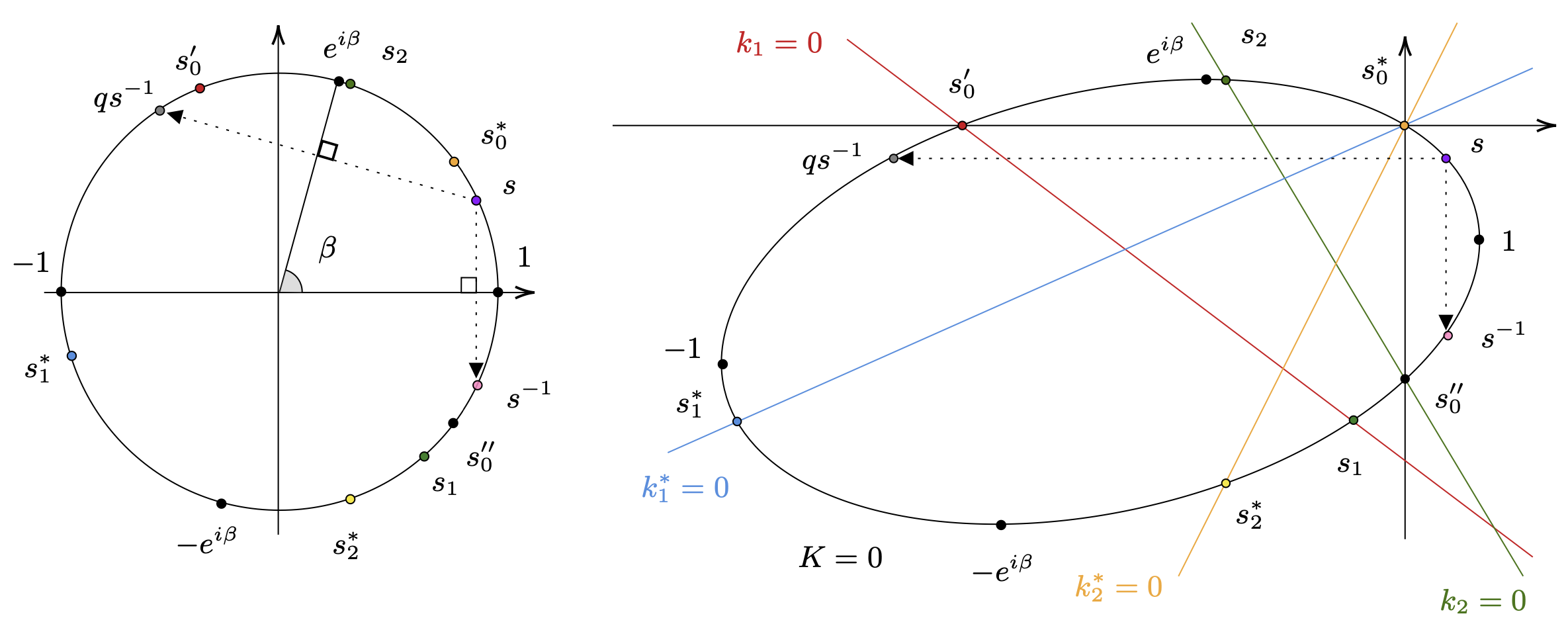}
\end{center}
\caption{On the left, the points introduced in Lemmas~\ref{lemma:k(s)} and~\ref{lemma:k*(s)} on the unit complex circle $\mathbb U$; on the right, their images by $(\boldsymbol{\mathrm{x}},\boldsymbol{\mathrm{y}})$ on the ellipse $\mathcal{S}\cap \mathbb R^2$. For the sake of readability, points on the ellipse are labelled with their preimages. Invariances in~\eqref{eq:invariance} are represented by the dotted arrows.}\label{fig:inv}
\end{figure}

\paragraph{Boundary Value Problem}
We define a hyperbola $\mathcal{H}$ deeply linked to the kernel by
$$
\mathcal{H}:= Y^\pm([x^+,\infty))=\{y\in\mathbb{C}: K(x,y)=0 \text{ and } x\in[x^+,\infty) \}.
$$
Noticing that $\boldsymbol{\mathrm{x}}(\mathbb R_+)=[x^+,+\infty)$ and by the invariance $\boldsymbol{\mathrm{y}}(s)=\boldsymbol{\mathrm{y}}(\boldsymbol{\mathrm{q}}s^{-1})$ we can see that
\begin{equation}\label{eq:Hy}
\mathcal{H}=\boldsymbol{\mathrm{y}}(\mathbb R_+)=\boldsymbol{\mathrm{y}}(\boldsymbol{\mathrm{q}}\mathbb R_+).
\end{equation}
This hyperbola is the boundary of the Boundary Value Problem stated below. We now define $\mathcal{G}_{\mathcal{H}}$ the domain of $\mathbb{C}$ bounded by $\mathcal{H}$ and containing $y^+$, see Figure~\ref{fig:GH}.
By~\eqref{eq:Hy}, remembering  that $\boldsymbol{\mathrm{q}}=e^{2i\beta}$ and $\boldsymbol{\mathrm{y}}^{-1}(y^+)=e^{i\beta}$ we see that
    \begin{equation}\label{eq:GHunif}
    \overline{\mathcal{G}}_\mathcal{H}=\boldsymbol{\mathrm{y}}\Big(\{ae^{ib}, (a,b)\in \mathbb R_+\times [0,2\beta]\}\Big) ,
    \end{equation}
see Figure~\ref{fig:general}. Finally, we compute
\begin{equation}\label{eq:b1}
\boldsymbol{\mathrm{y}}(s_1)=- \frac{2(r_1\mu_1+\mu_2)}{\sigma_{22}+\sigma_{11}r_1^2+2\sigma_{12} r_1}.
\end{equation}

The following proposition is a Carleman Boundary Value Problem which characterizes the Laplace transform $\varphi_1$ and which can be easily obtained from the functional equation~\eqref{eq:functionaleq}, see~\cite{ErFrHu-20}.
\begin{proposition}[Proposition 22 and Lemma~32 in \cite{ErFrHu-20}]The Laplace transform $\varphi_1$ satisfies the boundary value problem:
\begin{enumerate}
\item $\varphi_1$ is meromorphic on the open domain $\mathcal{G}_{\mathcal{H}}$ and continuous on $\overline{\mathcal{G}}_{\mathcal{H}}:=\mathcal{G}_\mathcal{H}\cup\mathcal{H}$;
\item \label{eq:poleBVP} $\varphi_1$ admits one or two poles in $\overline{\mathcal{G}}_{\mathcal{H}}$, $0$ is always a simple pole of $\varphi_1$ and $\boldsymbol{\mathrm{y}}(s_1)$ is a simple pole of $\varphi_1$ if and only if $2\varepsilon+\theta\geqslant 2\pi$;
\item for some positive constant $C$ the asymptotics of $\varphi_1$ when $y\to\infty$ is given by
\begin{equation}\label{eq:asympphi}
\varphi_1 (y)  \sim C y^{-\alpha-1};
\end{equation}
\item $\varphi_1$ satisfies the boundary condition
\begin{equation}\label{eq:BVP1}
    \varphi_1(\overline{y})=G(y)\varphi_1(y),\quad\quad \forall y\in\mathcal{H}
\end{equation}
where
\begin{equation}\label{eq:BVP2}
    \nb{G(y)=\frac{k_1(X^+(y),y)}{k_2(X^+(y),y)} \frac{k_2(X^+(y),\overline{y})}{k_1(X^+(y),\overline{y})}.}
\end{equation}
\end{enumerate}
\label{prop:BVP} 
\end{proposition}
In the next section, our strategy will be to find cases where the function $G$ simplifies in order to find rational and D-algebraic solutions to this Boundary Value Problem.


\section{Tutte's invariants and Laplace transform}
\label{sec:tutte}

\subsection{Decoupling and Tutte's invariant}

\paragraph{Heuristic of the method}

The method involves finding all cases where there exists decoupling in the following sense. Recall that $\mathcal{S}=\{(x,y)\in\mathbb C^2 : K(x,y)=0\}$.
\begin{definition}[Decoupling]
A pair of rational functions $(P,Q)$ satisfying
\begin{equation}\label{eq:decouple}
    \frac{k_2(x,y)}{k_1(x,y)}=\lambda\frac{P(x)}{Q(y)}\text{ for all }(x,y)\in\mathcal{S}
\end{equation}
for some constant $\lambda$ is called a \nb{\textit{decoupling pair}}. 
\label{def:decoupling}
\end{definition}
We shall see that the existence of a decoupling pair leads to the study of what is called an invariant, which glues together the upper and the lower branches of the hyperbola $\mathcal{H}$ in the following sense.
\begin{definition}[Invariant]
A function $I$ which is meromorphic in $\mathcal{G}_{\mathcal{H}}$, continuous on its boundary $\mathcal{H}$ and satisfying $I(\overline{y})=I(y)$ for all $y\in \mathcal{H}$ is called an \nb{\textit{invariant}}. 
\label{def:inv}
\end{definition}

Under the existence of a decoupling pair $(P,Q)$, the boundary condition (\ref{eq:BVP1}) can be rewritten as
\begin{equation}\label{eq:Qphi1}
Q\varphi_1(\overline{y})=Q\varphi_1(y),\quad\quad \forall y\in \mathcal{H}.
\end{equation} 
If $(P,Q)$ is a decoupling pair,
the function $Q\varphi_1$ is then called the unknown invariant (since we are looking for $\varphi_1$). 
We now introduce a conformal \textit{gluing} function $w$, which we call the canonical invariant, in terms of a classical Gauss hypergeometric function which is often called \textit{generalized Chebyshev polynomial},
\begin{equation*}
w(y):=\!\!\!\phantom{x}_2F_1\left(-\frac{\pi}{\beta},\frac{\pi}{\beta};\frac{1}{2};\frac{1}{2}\left(1-\frac{2y-(y^++y^-)}{y^+-y^-}\right)\right)
=\cos\left(\frac{\pi}{\beta}\arccos\left(\frac{2y-(y^++y^-)}{y^+-y^-}\right)\right) .
\end{equation*}
The fact that $w$ is a conformal invariant (in the sense of Definition \ref{def:inv}) is proven in Lemma~5.3 of~\cite{BoMe-El-Fr-Ha-Ra}. In particular, $w$ is analytic and bijective from $\mathcal{G}_{\mathcal{H}}$ to $\mathbb C\setminus (-\infty,-1]$ and
\begin{equation}
w(\overline{y})=w(y),\quad\quad \forall y \in\mathcal{H}.
\label{eq:winv}
\end{equation}
\nb{Equation~\eqref{eq:winv} justifies the terminology conformal \textit{gluing} function for $w$, often encountered in the literature on Tutte's invariant method.}

Lemma 5.3 of~\cite{BoMe-El-Fr-Ha-Ra} also provides, for some constant $\hat{C}$, the asymptotics
\begin{equation}\label{eq:asympw}
w(y)\sim  \hat{C}y^{\pi/\beta}.
\end{equation}
It is also well known that $w$ is a polynomial when $\pi/\beta\in\mathbb{Z}$, is algebraic when $\pi/\beta\in\mathbb{Q}$ and is always D-finite. Remark that the set of D-finite functions is stable by multiplication but not by division and $1/w$ is D-algebraic but not necessarily D-finite. See Proposition 5.2 of~\cite{BoMe-El-Fr-Ha-Ra}.

The key point of Tutte's invariant method is to express the unknown invariant in terms of the canonical invariant. The following crucial lemma shows that there are few invariants.
\begin{lemma}[Invariant lemma]\label{lemma:Tutte}
If $I$ is an invariant in the sense of Definition~\ref{def:inv} which doesn't have any pole on $\overline{\mathcal{G}}_\mathcal{H}$ and has a finite limit at $+\infty$, then $I$ is constant. 
\end{lemma}

\begin{proof}
Since $w$ is conformal, and maps $\mathcal{G}_{\mathcal{H}}$ to the cut plane $\mathbb C\setminus (-\infty,-1]$, $I \circ w^{-1}$ is analytic on this cut plane, and continuous on the cut thanks to (\ref{eq:winv}). By Morera's theorem, $I \circ w^{-1}$ (or better yet, its continuous extension) is meromorphic on $\mathbb C$. Since $I$ is bounded, Liouville's theorem implies that $I \circ w^{-1}$ (and then $I$) is constant.
\end{proof}

\paragraph{Decoupling}

\begin{figure}[t]
\centering
\includegraphics[width=0.7\linewidth]{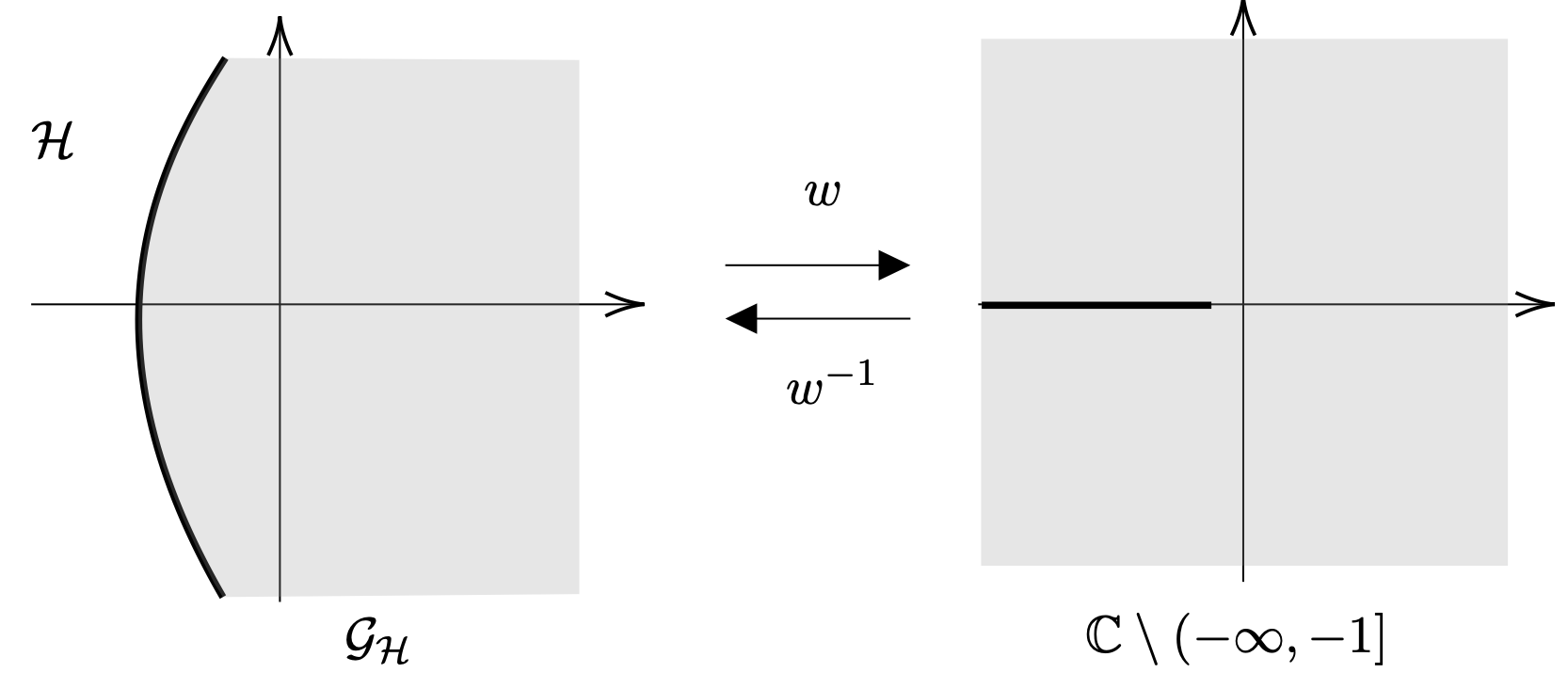}
\caption{Domains and codomains for $w$ and its inverse.}
\label{fig:GH}
\end{figure}

In the following proposition, we obtain a necessary and sufficient condition for the existence of a decoupling condition \eqref{eq:decouple} and we give explicit decoupling pairs.
\begin{proposition}[Decoupling condition]
There exists a rational decoupling pair $(P,Q)$ if and only if 
$$\alpha\in\mathbb Z+\frac{\pi}{\beta}\mathbb Z.$$ In this case let $d,r\in\mathbb{Z}$ such that $\alpha=d-1+r\pi/\beta$, \textit{i.e.} $\delta+\varepsilon=(d-1)\beta + (r+1)\pi$. Then $d$ cannot be equal to $1$ and we distinguish two cases:
    \begin{itemize}
    \item If $d\geqslant 2$, then one can choose the following polynomial decoupling pair,
    \begin{equation}\label{eq:PQ}
        P(x):=x \prod_{k=0}^{d-2} \frac{x-\boldsymbol{\mathrm{x}}(s_2\boldsymbol{\mathrm{q}}^k)}{-\boldsymbol{\mathrm{x}}(s_2\boldsymbol{\mathrm{q}}^k)},
        \quad\quad 
        Q(y):=y\prod_{k=0}^{d-2}\frac{y-\boldsymbol{\mathrm{y}}\left(s_1/\boldsymbol{\mathrm{q}}^k\right)}{-\boldsymbol{\mathrm{y}}\left(s_1/\boldsymbol{\mathrm{q}}^k\right)}.
    \end{equation}
    \item If $d\leqslant 0$, then one can choose the following rational decoupling pair,
    \begin{equation}\label{eq:PQ2}
        P(x):=x \prod_{k=1}^{1-d} \frac{-\boldsymbol{\mathrm{x}}(s_2\boldsymbol{\mathrm{q}}^k)^{-1}}{x-\boldsymbol{\mathrm{x}}(s_2/\boldsymbol{\mathrm{q}}^{k})},
        \quad\quad 
        Q(y):=y\prod_{k=1}^{1-d}\frac{-\boldsymbol{\mathrm{y}}\left(s_1/\boldsymbol{\mathrm{q}}^k\right)^{-1}}{y-\boldsymbol{\mathrm{y}}\left(s_1\boldsymbol{\mathrm{q}}^k\right)}.
    \end{equation}
    \end{itemize}
We have $d=\mathrm{deg}(P)=\mathrm{deg}(Q)$.
\label{prop:decoupl}
\end{proposition}

\begin{proof} If there exists a rational decoupling $(P,Q)$, see \eqref{eq:decouple}, then by the invariances of $\boldsymbol{\mathrm{x}}$ and $\boldsymbol{\mathrm{y}}$ established in \eqref{eq:invariance} we have
\begin{equation}
    \frac{k_2(\boldsymbol{\mathrm{x}}(s),\boldsymbol{\mathrm{y}}(s))k_1(\boldsymbol{\mathrm{x}}(s^{-1}),\boldsymbol{\mathrm{y}}(s^{-1}))}{k_1(\boldsymbol{\mathrm{x}}(s),\boldsymbol{\mathrm{y}}(s))k_2(\boldsymbol{\mathrm{x}}(s^{-1}),\boldsymbol{\mathrm{y}}(s^{-1}))}
    =\frac{P(\boldsymbol{\mathrm{x}}(s))Q(\boldsymbol{\mathrm{y}}(s^{-1}))}{Q(\boldsymbol{\mathrm{y}}(s))P(\boldsymbol{\mathrm{x}}(s^{-1}))}
    =\frac{Q(\boldsymbol{\mathrm{y}}(\boldsymbol{\mathrm{q}}s))}{Q(\boldsymbol{\mathrm{y}}(s))} .
    \label{eq:lim1}
\end{equation}
According to Lemma~\ref{lemma:k(s)} we also have
\begin{equation}
    \frac{k_2(\boldsymbol{\mathrm{x}}(s),\boldsymbol{\mathrm{y}}(s))k_1(\boldsymbol{\mathrm{x}}(s^{-1}),\boldsymbol{\mathrm{y}}(s^{-1}))}{k_1(\boldsymbol{\mathrm{x}}(s),\boldsymbol{\mathrm{y}}(s))k_2(\boldsymbol{\mathrm{x}}(s^{-1}),\boldsymbol{\mathrm{y}}(s^{-1}))}=\frac{(s-s_0'')(s-s_2)}{(s-s_0')(s-s_1)}\cdot\frac{(s^{-1}-s_0')(s^{-1}-s_1)}{(s^{-1}-s_0'')(s^{-1}-s_2)} .
     \label{eq:lim2}
\end{equation}
Taking the limit as $s$ goes to $+\infty$ of~\eqref{eq:lim1} and~\eqref{eq:lim2}, we get 
\begin{equation}\label{eq:lim}
    \frac{s_0's_1}{s_0''s_2}=\boldsymbol{\mathrm{q}}^{d},\quad\text{where }d:=\text{deg}(Q)\in\mathbb{Z}.
\end{equation}
Plugging in the values of $s_0'$, $s_0''$, $s_1$, $s_2$ obtained in Lemma~\ref{lemma:k(s)} and remembering that $\boldsymbol{\mathrm{q}}=e^{2i\beta}$ in~\eqref{eq:lim} we obtain that $e^{2i(\beta+\delta+\varepsilon)}=e^{2id\beta}$ and then there exists $r\in\mathbb Z$ such that
\begin{equation}\label{eq:decouplecondition}
    \alpha=\frac{\delta +\varepsilon -\pi}{\beta}=d-1 + r\frac{\pi}{\beta}.
\end{equation}
Conversely, we assume that~\eqref{eq:decouplecondition} holds. First, let us treat the case $d\geqslant 2$ and assume that $P$ and $Q$ are given by~\eqref{eq:PQ}. Through the uniformization~\eqref{eq:unif}, one gets
    \begin{equation*}
    P(\boldsymbol{\mathrm{x}}(s))=\frac{1}{s^{d}}(s-s_0'')\left(s-\frac{1}{s_0''}\right)\prod_{k=0}^{d-2}\frac{s-s_2\boldsymbol{\mathrm{q}}^k}{-\boldsymbol{\mathrm{x}}(s_2\boldsymbol{\mathrm{q}}^k)}\left(s-\frac{1}{s_2\boldsymbol{\mathrm{q}}^k}\right),
    \end{equation*}
    \begin{equation*}
    Q(\boldsymbol{\mathrm{y}}(s))=\frac{1}{s^{d}}(s-s_0')\left(s-\frac{\boldsymbol{\mathrm{q}}}{s_0'}\right)\prod_{k=0}^{d-2}\frac{s-s_1/\boldsymbol{\mathrm{q}}^k}{-\boldsymbol{\mathrm{y}}\left(s_1/\boldsymbol{\mathrm{q}}^k\right)}\left(s-\frac{\boldsymbol{\mathrm{q}}^{k+1}}{s_1}\right).
    \end{equation*}
   We know that $s_0'=\boldsymbol{\mathrm{q}}s_0''$. Given that $\alpha=d-1+r\pi/\beta$ with $d$ and $r$ in $\mathbb Z$, we can also use the fact that $s_1=\boldsymbol{\mathrm{q}}^{d-1}s_2$. When taking the ratio of $P(\boldsymbol{\mathrm{x}}(s))$ and $Q(\boldsymbol{\mathrm{y}}(s))$, these identities produce a telescoping which gives
    \begin{equation*}
    \frac{P(\boldsymbol{\mathrm{x}}(s))}{Q(\boldsymbol{\mathrm{y}}(s))}
    = \prod_{k=0}^{d-2}\frac{\boldsymbol{\mathrm{y}}(s_1/\boldsymbol{\mathrm{q}}^k)}{\boldsymbol{\mathrm{x}}(s_2\boldsymbol{\mathrm{q}}^k)}
    \frac{(s-s_0'')(s-s_2)}{(s-s_0')(s-s_1)}
    = \frac{1}{\lambda}
  \frac{  k_2(\boldsymbol{\mathrm{x}}(s),\boldsymbol{\mathrm{y}}(s))}{k_1(\boldsymbol{\mathrm{x}}(s),\boldsymbol{\mathrm{y}}(s))},
    \end{equation*}
where the last equality comes from Lemma~\ref{lemma:k(s)} and taking
\begin{equation}\label{eq:lambda}
\lambda:=\frac{C_1}{C_2}\prod_{k=0}^{d-2}\frac{\boldsymbol{\mathrm{x}}(s_2\boldsymbol{\mathrm{q}}^k)}{\boldsymbol{\mathrm{y}}(s_1/\boldsymbol{\mathrm{q}}^k)} .
\end{equation}
We deduce that $(P,Q)$ is a decoupling pair.
The proof is similar for the case $d\leqslant 0$. The fact that $d$ cannot be equal to $1$ directly derives from the fact that $\varepsilon,\delta\in (0,\pi)$, $\beta\in(0,\pi)$ and $\alpha\geqslant 1$.
\end{proof}

\begin{lemma}[Simple root condition]\label{lemma:double}
Let $\alpha\in \mathbb Z+\frac{\pi}{\beta}\mathbb Z$ and $(d,r)\in\mathbb{Z}^2$ such that $\alpha=d-1+r\pi/\beta$. If $\beta/\pi\in\mathbb{Q}$ then $(d,r)$ is not unique, in this case for $\beta/\pi=p/q$ for $p$ and $q$ relatively prime and $p<q$, we (can) choose $d$ such that $|d|<q$. If $d\geqslant 2$ (resp. $d\leqslant 0$) then $P$ and $Q$ have no multiple roots (resp. pole) if and only if for all $k\in\{1,\dots,2d-4\}$ (resp. $k\in\{2d-1,\dots,-2\}$) we have
\begin{equation}
\theta-2\delta+k\beta\not\equiv 0\;\mathrm{mod}(\pi).
\label{eq:doubleroot} 
\end{equation}
\end{lemma}

\begin{proof}
Let $d \geqslant 2$. The polynomial $P$ has a double root (or more) if and only if for some $i\ne j\in\{0,\dots, d-2 \}$ we have $\boldsymbol{\mathrm{x}}(s_2\boldsymbol{\mathrm{q}}^i)=\boldsymbol{\mathrm{x}}(s_2\boldsymbol{\mathrm{q}}^j)$. Using the expression of $\boldsymbol{\mathrm{x}}$ given in~\eqref{eq:unif} there are only two ways for this to happen. The first one is that $\boldsymbol{\mathrm{q}}^i=\boldsymbol{\mathrm{q}}^j$, \textit{i.e.} $i\beta=j\beta \ \mathrm{mod}(\pi)$ which is not possible even when $\beta/\pi=p/q\in\mathbb{Q}$ since $|i-j|<|d|<q$. The second one is that $s_2\boldsymbol{\mathrm{q}}^i=(s_2\boldsymbol{\mathrm{q}}^j)^{-1}$, using the value of $s_2$ in~\eqref{eq:si} it is equivalent to $\theta-2\delta+(i+j)\beta\equiv 0 \;\text{mod}(\pi)$. Similarly, $Q$ has a double root if and only if $\theta-2\delta+(2d-i-j-3)\beta\equiv 0\;\text{mod}(\pi)$. The case $d\leqslant 0$ is similar.
\end{proof}
One should observe that the decoupling condition $\alpha\in \mathbb Z+\frac{\pi}{\beta}\mathbb Z$ doesn't depend on $\theta$ (and hence on the drift) while the multiple root condition does.

\subsection{Explicit expression for the Laplace transforms}

We now state our first main result when $\alpha\in\mathbb N$.
\begin{theorem}[Laplace transforms, $\alpha\in\mathbb{N}$]
If $\alpha\in\mathbb N$ then 
the rational function defined by 
\begin{equation}\label{eq:L}
L(x,y):=\frac{k_1(x,y)P(x)+k_2(x,y)Q(y)}{K(x,y)}
\end{equation}
is a polynomial and we have
$$
\varphi_1(y)=\frac{1}{Q(y)},
\quad\quad 
\varphi_2(x)=\frac{1}{P(x)},
\quad\quad 
\varphi(x,y)=\frac{L(x,y)}{P(x)Q(y)}
$$
where $P$ and $Q$ are given in~\eqref{eq:PQ}.
\label{thm:laplace}
\end{theorem}
\begin{proof}
Assuming that $\alpha=d-1\in\mathbb N$, the polynomials $P$ and $Q$ given in~\eqref{eq:PQ} form a decoupling pair by Proposition~\ref{prop:decoupl}. The boundary value problem of Proposition~\ref{prop:BVP} thus implies that $Q\varphi_1$ is an invariant, see~\eqref{eq:Qphi1}.
According to Lemma~\ref{lemma:Tutte}, we only need to prove that $Q\varphi_1$ doesn't have any pole on $\overline{\mathcal{G}}_{\mathcal{H}}$, and has a finite limit as $y$ goes to $+\infty$. By~(\ref{eq:asympphi}) and~(\ref{eq:PQ}), and since $\alpha=d-1$ we have
$$Q(y)\varphi_1(y)\sim \frac{y^d}{\prod_{k=0}^{d-2}-\boldsymbol{\mathrm{y}}(s_1/\boldsymbol{\mathrm{q}}^k)}
\cdot Cy^{-\alpha-1}
=\frac{C}{\prod_{k=0}^{d-2}-\boldsymbol{\mathrm{y}}(s_1/\boldsymbol{\mathrm{q}}^k)}.$$
Furthermore the poles of $\varphi_1$ given in Proposition~\ref{prop:BVP} \textit{i.e.} $0$ and $\boldsymbol{\mathrm{y}}(s_1)$ when $2\varepsilon+\theta\geqslant 2\pi$, are compensated by the zeros of $Q$. Indeed $0$ and $\boldsymbol{\mathrm{y}}(s_1)$ are always roots of $Q$. By Lemma~\ref{lemma:Tutte}, there exists $\kappa\in\mathbb C$ such that $Q\varphi \equiv\kappa$. On the one hand, using the fact that $Q'(0)=1$ gives
$$\lim_{y\to 0}y\varphi_1(y)=\kappa\lim_{y\to 0}\frac{y}{Q(y)}=\frac{\kappa}{Q'(0)}=\kappa.$$
On the other hand, by the final value theorem and~\eqref{eq:limits} we have
\begin{equation}
\lim_{y\to 0}y\varphi_1(y)=\lim_{v\to +\infty}\mathbb P_{(0,v)}(T=+\infty)=1.
\label{eq:tvf}
\end{equation}
Hence $\kappa=1$ and $\varphi_1=1/Q$. The same method also works to show that $\varphi_2=1/P$. 
Replacing $\varphi_1$ and $\varphi_2$ in the functional equation~\eqref{eq:functionaleq}, one can obtain
\begin{equation}
\frac{k_2(x,y)}{k_1(x,y)}=-\frac{P(x)}{Q(y)},\quad \text{for all }(x,y)\in\mathcal{S}.
\end{equation}
Comparing with Definition~\ref{def:decoupling} we can see that the constant $\lambda$ given in \eqref{eq:lambda} is equal to $-1$.
The polynomial $k_2Q+ k_1P$ vanishes on $\mathcal{S}$ the set of the zeros of $K$. By Hilbert's Nullstellensatz~\cite{hartshorne}, 
$$k_2Q+ k_1 P\in \sqrt{(K)}$$
where 
$(K):=\{LK, L\in\mathbb C[X,Y]\}\text{ and }\sqrt{(K)}:=\{H\in\mathbb C[X,Y] : \exists m \in \mathbb N\text{ s.t. }H^m\in (K)\}$
are the ideal generated by $K$ and its radical. \nb{By Lemma~\ref{lemma:irreducible},} $K$ is irreducible 
which implies that $\sqrt{(K)}=(K)$. This shows that there exists $L\in \mathbb C[X,Y]$ such that
$L(x,y)K(x,y)= k_1(x,y)P(x)+k_2(x,y)Q(y).$
We see that $L(x,y)\in \mathbb R$ for all $(x,y)\in\mathbb R^2$, so the coefficients of $L$ must be real. Substituting the values for $\varphi_1$ and $\varphi_2$ into the functional equation \eqref{eq:functionaleq} yields $\varphi(x,y)=L(x,y)/(P(x)Q(y))$.
\end{proof}

We now state a lemma useful to obtain our second main result which deals with the case where $\alpha\in\mathbb Z+\frac{\pi}{\beta}\mathbb Z$. We consider $d$ and $r\in\mathbb{Z}$ such that $\alpha=(d-1)+r\pi/\beta$. 
If $\beta/\pi=p/q\in\mathbb{Q}$, with $p$ and $q$ relatively prime, we (can) choose $|d|<q$.
For further use, we need to study the number of zeros and poles of $Q\varphi_1$ which belong to $\overline{\mathcal{G}}_\mathcal{H}$. First of all, we can see that $0$ is always a root of $Q$ and a pole of $\varphi_1$, which therefore compensate each other considering $Q\varphi_1$. 

When $d \geqslant 2$ we denote 
\begin{equation}
\mathrm{Z}:= \{\boldsymbol{\mathrm{y}}(s_1/\boldsymbol{\mathrm{q}}^k ) \in \overline{\mathcal{G}}_\mathcal{H}: k=1,\dots,d-2 \}
\label{eq:Zdef}
\end{equation} 
which is a set containing (all the) zeros of $Q\varphi_1$ in $\overline{\mathcal{G}}_\mathcal{H}$, 
see~\eqref{eq:PQ}. Note that in this definition $k$ cannot be taken equal to $0$ since when $\boldsymbol{\mathrm{y}}(s_1)$ is both a pole of $\varphi_1$ in $\overline{\mathcal{G}}_\mathcal{H}$ and a zero of $Q$ they compensate each other, by item~\ref{eq:poleBVP} of Proposition~\ref{prop:BVP}. 

When $d \leqslant 0$ we denote 
\begin{equation}
\mathrm{P}:=\{\boldsymbol{\mathrm{y}}(s_1 \boldsymbol{\mathrm{q}}^k) \in \overline{\mathcal{G}}_\mathcal{H} : k=0,\dots,1-d \}
\label{eq:Pdef}
\end{equation}
which is the set of poles of $Q\varphi_1$ in $\overline{\mathcal{G}}_\mathcal{H}$,  
see~\eqref{eq:PQ2}. Note that in this definition $k$ can be taken equal to $0$ since $\boldsymbol{\mathrm{y}}(s_1)$ is a pole of $\varphi_1$ which can belongs to $\overline{\mathcal{G}}_\mathcal{H}$, see item~\ref{eq:poleBVP} of Proposition~\ref{prop:BVP}. See Figure~\ref{fig:general} to visualize $\mathrm{P}$ and $\mathrm{Z}$.

\begin{figure}[t]
\centering
\includegraphics[width=0.78\linewidth]{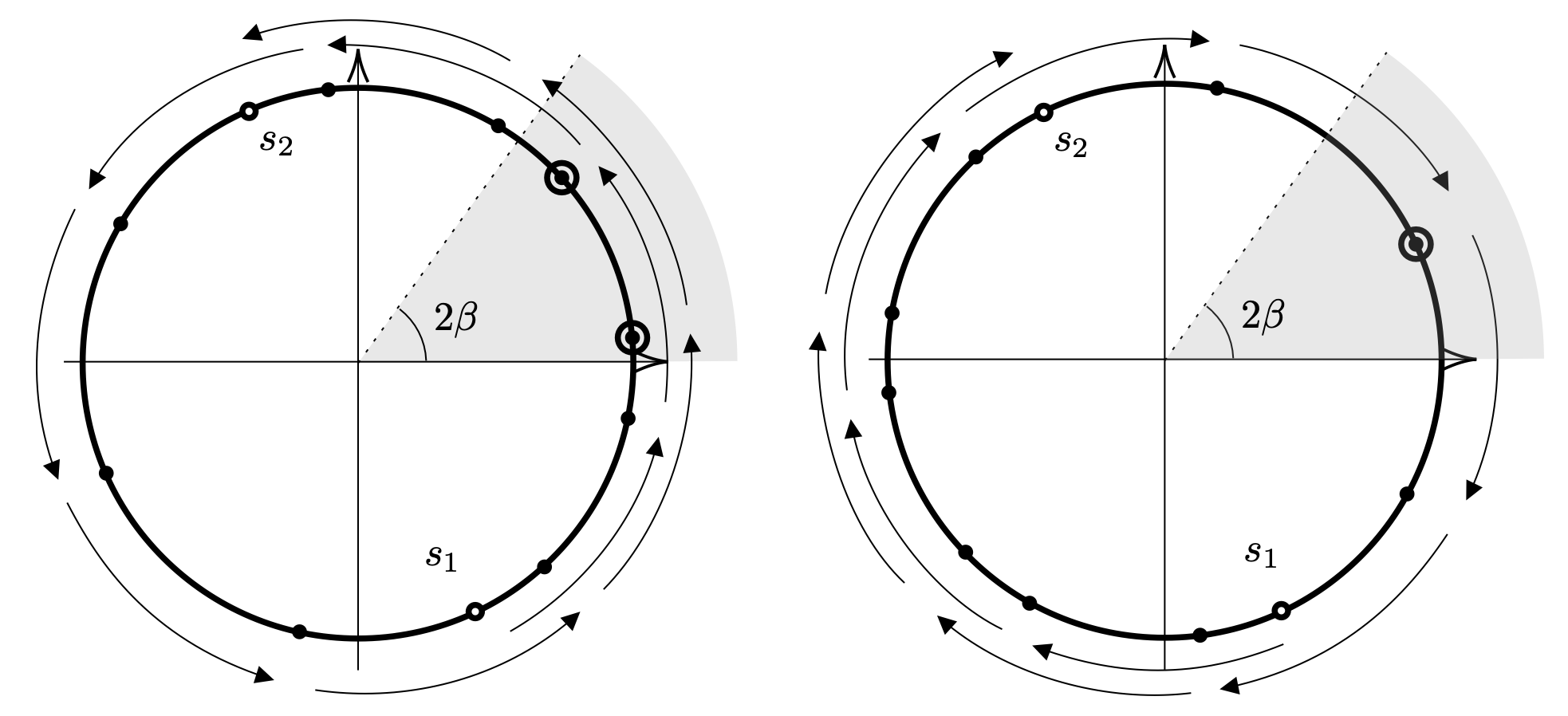}
\caption{Enumeration of $\mathrm{P}$ and $\mathrm{Z}$ for $(d,r)=(-9,2)$ (on the left) and $(d,r)=(11,-1)$ (on the right). The shaded area corresponds to $\boldsymbol{\mathrm{y}}^{-1}(\overline{\mathcal{G}}_\mathcal{H})$.}
\label{fig:general}
\end{figure}

\begin{lemma}[Cardinal of $\mathrm{Z}$ and $\mathrm{P}$]
Let $\alpha\geqslant 1$ and assume that $\alpha\in\mathbb Z+\frac{\pi}{\beta}\mathbb Z$ and~\eqref{eq:doubleroot} holds. Let $d$ and $r\in\mathbb{Z}$ such that $\alpha=(d-1)+r\pi/\beta$, \textit{i.e.} $(d-1)\beta+(r+1)\pi=\delta+\varepsilon$. Then $d\neq 1$ and we have 
\begin{enumerate}
\item[(i)] If $d\geqslant 2$ then $r\leqslant 0$ and we have
\begin{equation*}
    \mathrm{Card}(\mathrm{Z})= -r. 
     \end{equation*}
\item[(ii)] If $d\leqslant 0$ then $r>0$ and we have
 \begin{equation*}
     \mathrm{Card}(\mathrm{P})= r. 
     \end{equation*}
\end{enumerate}
\label{lemma:count}
\end{lemma}
\begin{proof}
Using the fact that $\varepsilon,\delta\in (0,\pi)$, $\beta\in(0,\pi)$ and $\alpha\geqslant 1$ it is easy to see that $d$ cannot be equal to $1$, that $d\geqslant 2$ implies $r\leqslant 0$ and that $d<1$ implies $r>0$.  
\begin{enumerate}
\item[(i)] Assume that $d\geqslant 2$ and $r\leqslant 0$.
Recalling equation~\eqref{eq:GHunif} and noticing that for all $k\in\mathbb{Z}$, $s_1/\boldsymbol{\mathrm{q}}^k\in \mathbb U$ (where $\mathbb U$ is the complex unit circle) we define
\begin{equation*}
C_0:=\{e^{ib}, b\in[0,2\beta]\}=\boldsymbol{\mathrm{y}}^{-1}(\overline{\mathcal{G}}_\mathcal{H})\cap \mathbb U
\end{equation*}
and we have
$$\text{Card}(\mathrm{Z})=\text{Card}\Big(\{s_1/\boldsymbol{\mathrm{q}}^k \in C_0 : k=1,\dots , d-2\}\Big).$$
We recall that $s_1=\boldsymbol{\mathrm{q}}^{d-1}s_2$, and so we need to count the number of points $s_1/\boldsymbol{\mathrm{q}}^k$ for $k=1,\dots, d-2$ which have their argument in $(0,2\beta)$ modulo $2\pi$.
These points can be obtained by making $d-1$ successive rotations of angle $-2\beta$, starting from $s_1$ to $s_2$ (without taking into account $s_1$ and $s_2$). By doing this, the number of complete revolutions around the unit circle in the clockwise direction is $-r$. This comes from the fact that, denoting $\arg s_1=\theta+2\varepsilon\in(\theta,\theta+2\pi)$ and $\arg s_2=\theta-2\delta+2\pi\in(\theta,\theta+2\pi)$, we have 
$$\arg s_2 - \arg s_1=(d-1)(-2\beta) - r (2\pi)>0.$$ 
Since $0<\theta<\beta$, there are exactly $-r$ points $s_1/\boldsymbol{\mathrm{q}}^k$ for $k=1,\dots, d-2$ which have their argument in $(0,2\beta)$ modulo $2\pi$, see Figure~\ref{fig:general}. Interested readers may refer to the study of mechanical or Sturmian  sequences~\cite{Loth-02} where this kind of counting problem is standard.

\item[(ii)] The case $d\leqslant 0$ and $r>0$ is similar 
considering $1-d$ successive rotations of angle $2\beta$ from $s_1$ to $s_2$ making $r$ turn around the unit circle in the counter-clockwise direction.
\end{enumerate}
\end{proof}

We now state our second main result about $\varphi_1$ when $\alpha\in\mathbb Z+\frac{\pi}{\beta}\mathbb Z$.
A symmetrical result holds for $\varphi_2$, and $\varphi$ can thus be determined by~\eqref{eq:functionaleq}.
\begin{theorem}[Laplace transforms, $\alpha\in\mathbb Z+\frac{\pi}{\beta}\mathbb Z$]
Assume that $\alpha\in\mathbb Z+\frac{\pi}{\beta}\mathbb Z$ and $\alpha \geqslant 1$ and the simple root condition~\eqref{eq:doubleroot} holds. Let $d$ and $r\in\mathbb{Z}$ such that $\alpha=(d-1)+r\pi/\beta$, \textit{i.e.} $(d-1)\beta+(r+1)\pi=\delta+\varepsilon$, then
\begin{equation*}
\varphi_1(y)=\frac{S(w(y)) }{Q(y)},
\end{equation*} 
where $S$ is a rational function of degree $-r$ given by
\begin{equation}
    S(z):=\displaystyle\prod_{q\in \mathrm{Z}}\frac{z-w(q)}{w(0)-w(q)}, \text{ if } d\geqslant 2
    \quad \text{and} \quad
    S(z):=\displaystyle\prod_{p\in \mathrm{P}}\frac{w(0)-w(p)}{z-w(p)}, \text{ if } d\leqslant 0 .
    \label{eq:Sdef}
\end{equation}
We deduce sufficient conditions for $\varphi_1$, $\varphi_2$ and $\varphi$ to belongs to the hierarchy~\eqref{eq:hierarchy}.
If $\alpha\in\mathbb{N}$ these Laplace transforms are rational, if $\pi/\beta\in\mathbb{Q}$ and $\alpha\in\mathbb Z+\frac{\pi}{\beta}\mathbb Z$ they are algebraic, if $\alpha\in\mathbb N+\frac{\pi}{\beta}\mathbb Z$ they are D-finite and if  $\alpha\in\mathbb Z+\frac{\pi}{\beta}\mathbb Z$ they are D-algebraic.
\label{thm:laplace2}
\end{theorem}

\begin{proof}
Recall the definitions of 
$\mathrm{Z}$ in \eqref{eq:Zdef}, $\mathrm{P}$ in \eqref{eq:Pdef} and $Q$ in~\eqref{eq:PQ} and \eqref{eq:PQ2}.
The function $ (Q\varphi_1)/(S\circ w)$ is continuous on $\mathcal{H}$ and meromorphic on $\mathcal{G}_\mathcal{H}$. By~\eqref{eq:Qphi1} and~\eqref{eq:winv}, we have for all $y\in\mathcal{H}$,
$$\frac{Q\varphi_1(\overline{y})}{S(w (\overline{y}))}=\frac{Q\varphi_1(y)}{S(w(y))}.$$
The function $(Q\varphi_1)/(S\circ w)$ is then an invariant in the sense of Definition \ref{def:inv}. Recall that $\text{deg}(Q)=d$, $\text{deg}(S)=-r$ by Lemma~\ref{lemma:count}, $\varphi_1 (y)  \sim C y^{-\alpha-1}$ by \eqref{eq:asympphi}, $w(y)\sim  \hat{C}y^{\pi/\beta}$ by \eqref{eq:asympw}. For a constant $\kappa$ we obtain when $y\to\infty$,
\begin{equation*}
\frac{Q\varphi_1(y)}{S(w(y))}\sim \kappa  \frac{y^{d} y^{-\alpha-1}}{y^{-r \pi/\beta}} = \kappa
\end{equation*}
where the last equality comes from $\alpha=(d-1)+r\pi/\beta$.
By construction, $(Q\varphi_1)/(S\circ w)$ does not have any pole on $\overline{\mathcal{G}}_\mathcal{H}$. Indeed if $d\geqslant 2$ the roots of $Q$ compensate the poles of $1/(S\circ w)$ and if $d\leqslant 0$ the zeros of $1/(S\circ w)$ compensate the poles of $Q\varphi_1$. Then, the invariant Lemma~\ref{lemma:Tutte} assures that
$$\frac{Q\varphi_1(y)}{S(w(y))}=\kappa.$$
Applying again the final value theorem~\eqref{eq:tvf} and using the fact that $\lim_{y\to 0}Q(y)/y=Q'(0)=1$ and $S(w(0))=1$ we obtain the value of the constant: 
$ \kappa=\frac{Q(y)}{y} \frac{y\varphi_1(y)}{ S( w(y))}
\underset{y\to 0}{\longrightarrow} 1=\kappa.$
The sufficient conditions given in the theorem therefore follow from the properties of $w$ stated below~\eqref{eq:asympw}.
\end{proof}

\section{Absorption probability via compensation approach}
\label{sec:compensation}

This section deals with the case where $\alpha\in\mathbb{N}$. The aim is to show that the absorption probability is a sum of exponentials and to calculate precisely all the coefficients of this sum. To that end we invert the Laplace transforms and we explain the recursive compensation phenomenon which appears in this sum.

\paragraph{Inverse Laplace transform}

When $\alpha\in\mathbb{N}$ and the simple root condition~\eqref{eq:doubleroot} holds, we invert the Laplace transforms $\varphi_1$ and $\varphi_2$ obtained in Theorem~\ref{thm:laplace} by performing a partial fraction decomposition. Therefore, remembering that $\varphi_1$ and $\varphi_2$ are defined in~\eqref{eq:phi1phi2} as the Laplace transforms of the escape probability on the boundaries, the absorption probabilities starting from the boundaries can be written as sum-of-exponential and are explicitly given by
\begin{equation}\label{eq:univar}
\mathbb P_{(u,0)}(T<\infty)=\sum_{i=0}^{\alpha-1} d_i \exp\left(\boldsymbol{\mathrm{x}}(s_2\boldsymbol{\mathrm{q}}^i)u\right),\quad \mathbb P_{(0,v)}(T<\infty)=\sum_{j=0}^{\alpha-1} e_j\exp\left(\boldsymbol{\mathrm{y}}(s_1/\boldsymbol{\mathrm{q}}^j)v\right),
\end{equation}
where
\begin{equation*}
d_i=\frac{-1}{P'(\boldsymbol{\mathrm{x}}(s_2\boldsymbol{\mathrm{q}}^i))},\quad e_j=\frac{-1}{Q'(\boldsymbol{\mathrm{y}}(s_1/\boldsymbol{\mathrm{q}}^j))}.
\end{equation*}
However, inverting the bivariate Laplace transform $\varphi(x,y)$ and computing the coefficients involved is not immediately obvious. 
We now state the last main result of this article.

\begin{theorem}[Sum-of-exponential absorption probability]
Let $(Z_t)_{0 \leqslant t \leqslant T}$ a reflected Brownian motion in the quadrant of drift $\mu\in \mathbb{R}_+^2$, such that $\alpha\geqslant 1$, starting from $(u,v)\in \mathbb{R}_+^2$, where $T$ is the first hitting time of the vertex. We assume that for all $j\in\{1,\dots,2\alpha-2\}$, $\theta-2\delta+j\beta\not\equiv 0\;\mathrm{ mod}(\pi)$. The following statements are equivalent:
\begin{itemize}
\item[(i)] $\alpha=n$, for some integer $n\geqslant 1$  ;
\item[(ii)] there exist coefficients $a_1,\dots,a_{2n-1},b_1,\dots,b_{2n-1},c_1 , \dots , c_{2n-1}$ such that
\begin{equation}\label{eq:sumexp}
\mathbb{P}_{(u,v)}(T<\infty)= \sum_{k=1}^{2n-1} c_k \exp(a_ku+b_kv).
\end{equation}
\end{itemize}
In this case, the constants $a_k$ and $b_k$ are given by 
\begin{equation}\label{eq:akbk}
(a_{2k},b_{2k}) :=\left( \boldsymbol{\mathrm{x}}\left(s_1/ \boldsymbol{\mathrm{q}}^{k}\right)
,
\boldsymbol{\mathrm{y}}\left(s_1/\boldsymbol{\mathrm{q}}^{k}\right) \right)
\quad\text{and}\quad
(a_{2k+1},b_{2k+1}) :=\left( \boldsymbol{\mathrm{x}}\left(s_1/ \boldsymbol{\mathrm{q}}^{k+1}\right)
,
\boldsymbol{\mathrm{y}}\left(s_1/\boldsymbol{\mathrm{q}}^{k}\right)\right)
\end{equation}
and can also be computed thanks to the recurrence relationship stated in Proposition~\ref{rem:recakbk}.
The coefficients $c_k$ are determined by the recurrence relationship given in Proposition~\ref{prop:comp2}.
\label{thm:explicitproba}
\end{theorem}

\begin{proof}
First we assume that $\alpha=n$ and for all $j\in\{1,\dots,2\alpha-2\}$, $\theta-2\delta+j\beta\not\equiv 0\text{ mod}(\pi)$. \nb{Theorem~\ref{thm:laplace} gives an explicit expression of the Laplace transform: one can write $\varphi(x,y)=\frac{L(x,y)}{P(x)Q(y)}$ where $L$ is a polynomial given by~\eqref{eq:L} and perform a partial fraction decomposition of $\varphi$. }
It is \nb{then} possible to invert the Laplace transform. Recall the definition of $P$ and $Q$ given in~\eqref{eq:PQ} and remark that $\{s_2 \boldsymbol{\mathrm{q}}^i : i=0,\dots , \alpha-1 \}=\{s_1 / \boldsymbol{\mathrm{q}}^{i} : i=1,\dots , \alpha \}$ since $s_1=s_2 \mathrm{q}^\alpha$, we obtain $\alpha^2$ constants $\Tilde{c}_{i,j}$ such that
\begin{equation}
f(u,v)=\mathbb P_{(u,v)}(T<\infty)=\sum_{i=1}^{\alpha}\sum_{j=0}^{\alpha-1}\Tilde{c}_{i,j} \exp\Big(\boldsymbol{\mathrm{x}}(s_1 / \boldsymbol{\mathrm{q}}^{i})u+\boldsymbol{\mathrm{y}}(s_1/\boldsymbol{\mathrm{q}}^{j})v\Big).
\label{eq:falpha2}
\end{equation}
Actually, only $2\alpha-1$ of those constants $\Tilde{c}_{i,j}$ are non-zero. More precisely, we are now going to show that if $i\notin\{j,j+1\}$ then $\Tilde{c}_{i,j}= 0$.

Considering~\eqref{eq:falpha2}, the partial differential equation~\eqref{eq:PDE} leads to
$$0=\mathcal{G}f(u,v)=\sum_{i=1}^{\alpha}\sum_{j=0}^{\alpha-1}\Tilde{c}_{i,j}
K\Big(\boldsymbol{\mathrm{x}}(s_1 / \boldsymbol{\mathrm{q}}^{i}),\boldsymbol{\mathrm{y}}(s_1 / \boldsymbol{\mathrm{q}}^{j})\Big)
\exp\Big(\boldsymbol{\mathrm{x}}(s_1 / \boldsymbol{\mathrm{q}}^{i})u+\boldsymbol{\mathrm{y}}(s_1/\boldsymbol{\mathrm{q}}^{j})v\Big).$$
By linear independence of the exponential functions (the coefficients inside the exponentials are all different by~\eqref{eq:doubleroot}),
this implies that $\Tilde{c}_{i,j}=0$ when $K\Big(\boldsymbol{\mathrm{x}}(s_1 / \boldsymbol{\mathrm{q}}^{i}),\boldsymbol{\mathrm{y}}(s_1 / \boldsymbol{\mathrm{q}}^{j})\Big)\ne 0$. By~\eqref{eq:invariance} this must hold for all $(i,j)$ such that 
$s_1 / \boldsymbol{\mathrm{q}}^{i}\notin \ \{s_1 / \boldsymbol{\mathrm{q}}^{j},s_1 / \boldsymbol{\mathrm{q}}^{j+1}\}$
and then for $i\notin\{j,j+1\}$. 

For $i\in \{0,\dots , \alpha \}$ we set the constants $c_{2i}=\Tilde{c}_{i,i}$ and $c_{2i+1}=\Tilde{c}_{i,i-1}$ and we obtain~\eqref{eq:sumexp}. Proposition~\ref{prop:comp2} will give recurrence formulas satisfied by these constants.

Reciprocally, if the absorption probability $f(u,v) $ is a sum of exponentials then $f(0,v)$ is also a sum of exponentials where we denote $m$ the number of distinct exponentials in this sum. We deduce that $\varphi_1(y)$, which is the Laplace transform of $g(0,v)=1-f(0,v)$, is therefore equivalent up to a multiplicative constant to $y^{-m-1}$ when $y\to\infty$. We also know by~\eqref{eq:asympphi} that $\varphi_1(y)$ is equivalent up to a multiplicative constant to $y^{-\alpha-1}$, which implies that $\alpha=m\in\mathbb{N}$.
\end{proof}

In what follows it will be convenient to denote by $k_1^*$ and $k_2^*$ the following functions
\begin{equation}\label{eq:k*}
k_1^*(x,y):=(x,y)\cdot R^1=x-r_1 y,\quad\quad k_2^*(x,y):=(x,y)\cdot R^2=-r_2x+y
\end{equation}
as they naturally appear when applying $\partial_{R^1}$ and $\partial_{R^2}$, see~\eqref{eq:neumann}, to a function of the form $e^{au+bv}$.

\begin{lemma}\label{lemma:k*(s)} The functions $k_1$, $k_2$, $k_1^*$ and $k_2^*$ satisfy the following relations
\begin{equation*}
k_1\circ (\boldsymbol{\mathrm{x}},\boldsymbol{\mathrm{y}})(s)=-\frac{\sigma_{11}}{2} k_1^*\circ (\boldsymbol{\mathrm{x}},\boldsymbol{\mathrm{y}})(\boldsymbol{\mathrm{q}}s^{-1}),\quad\quad k_2\circ (\boldsymbol{\mathrm{x}},\boldsymbol{\mathrm{y}})(s)=-\frac{\sigma_{22}}{2} k_2^*\circ (\boldsymbol{\mathrm{x}},\boldsymbol{\mathrm{y}})(s^{-1}).
\end{equation*}
As a direct consequence we have, for some constants $C_1^*$ and $C_2^*$, for all $s\in\mathbb C^*$,
\begin{equation*}
k_1^*(\boldsymbol{\mathrm{x}}(s),\boldsymbol{\mathrm{y}}(s))=C_1^*\frac{(s-s_0^*)(s-s_1^*)}{s},\quad\quad k_2^*(\boldsymbol{\mathrm{x}}(s),\boldsymbol{\mathrm{y}}(s))=C_2^*\frac{(s-s_0^*)(s-s_2^*)}{s}
\end{equation*}
where $s_0^*:=s_0''^{-1}=\boldsymbol{\mathrm{q}}s_0'^{-1}$, $s_1^*:=\boldsymbol{\mathrm{q}}s_1^{-1}$ and $s_2^*:=s_2^{-1}$.
\end{lemma}
\begin{proof}
It is equivalent to prove that for all $x$ and $y$, $k_1(X^+(y),y)=-\frac{\sigma_{22}}{2}k_1^*(X^-(y),y)$ and $k_2(x,Y^+(x))=-\frac{\sigma_{11}}{2}k_2^*(x,Y^-(x))$ which can be easily verified using the definitions~\eqref{eq:Ypm}, \eqref{eq:kernel} and~\eqref{eq:k*}. Then, Lemma~\ref{lemma:k(s)} allows us to conclude. See Figure~\ref{fig:inv} for a geometric interpretation.
\end{proof}

The following proposition establishes a recurrence relationship which allows to compute $(a_k,b_k)$. It gives a very natural geometric interpretation of this sequence of points which belongs to the ellipse $\mathcal{E}:=\mathcal{S}\cap\mathbb{R}^2=\{(x,y)\in\mathbb{R}^2 :K(x,y)=0 \}$, starts at the intersection with the line $\{k_1^*=0\}$ and ends at the intersection with the line $\{k_2^*=0\}$. It can be visualized in Figure~\ref{fig:comp}. 

\begin{proposition}[Recursive relationship of the sequence $(a_k,b_k)$]
The sequence $(a_k,b_k)\in\mathcal{E}$ defined in~\eqref{eq:akbk} satisfies the following relations
\begin{equation*}
(a_{2k+1},b_{2k+1})=\left(\frac{b_{2k}}{a_{2k}}\cdot \frac{\sigma_{22}b_{2k}+2\mu_2}{\sigma_{11}},b_{2k}\right),\quad (a_{2k+2},b_{2k+2})=\left(a_{2k+1},\frac{a_{2k+1}}{b_{2k+1}}\cdot \frac{\sigma_{11}a_{2k+1}+2\mu_1}{\sigma_{22}}\right)
\end{equation*}
where $b_1=\boldsymbol{\mathrm{y}}\left(s_1\right)=\boldsymbol{\mathrm{y}}\left(s_1^*\right)$  was computed in~\eqref{eq:b1}, $a_1=\boldsymbol{\mathrm{x}}\left(s_1/\boldsymbol{\mathrm{q}} \right)=\boldsymbol{\mathrm{x}}\left(s_1^* \right)=r_1b_1$ and $k_1^*(a_1,b_1)=0$. We also have $k_2^*(a_{2\alpha-1},b_{2\alpha-1})=0$ and one can easily compute 
$$
a_{2\alpha-1}=\boldsymbol{\mathrm{x}}\left(s_2^*\right)=-\frac{2(\mu_1+r_2\mu_2)}{\sigma_{11}+\sigma_{22}r_2^2+2\sigma_{12}r_2}\text{ and }b_{2\alpha-1}=\boldsymbol{\mathrm{y}}\left(s_2^*\right)=r_2a_{2\alpha-1}.
$$
\label{rem:recakbk}
\end{proposition}
\begin{proof} By definition~\eqref{eq:akbk} we have $b_{2k}=b_{2k+1}$ and $a_{2k}\neq a_{2k+1}$. Furthermore $K(a_{2k},b_{2k})=K(a_{2k+1},b_{2k+1})=0$, then $a_{2k}$ and $a_{2k+1}$ must be the two distinct roots of the quadratic equation $K(x,b_{2k})=0$. Vieta's formula gives the value of the product of those roots in terms of the coefficients of the equation and we get the relation for $(a_{2k+1},b_{2k+1})$. The same method applies for the second relation about $(a_{2k+2},b_{2k+2})$.
\end{proof}

\begin{figure}[b]
\centering
\includegraphics[width=0.8\linewidth]{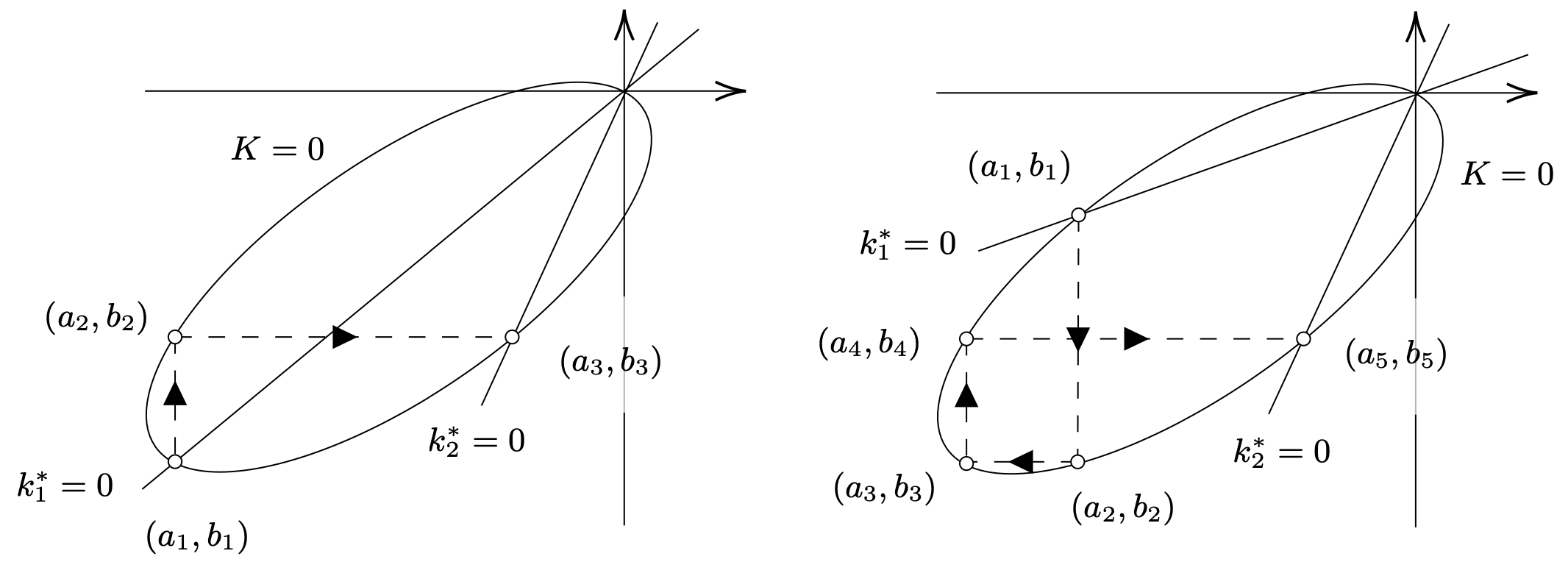}
\caption{Construction of the finite sequences $\{(a_i,b_i):1\leqslant i \leqslant 2\alpha -1\}$. On the left $\alpha =2$ while $\alpha=3$ on the right.}
\label{fig:comp}
\end{figure}

The aim is now to compute explicitly the coefficients $c_k$ which appear in Theorem~\ref{thm:explicitproba}.




\begin{proposition}[Recursive relationship of the sequence $c_k$]\label{prop:comp2}
We assume $\alpha\in\mathbb{N}$ and condition~\eqref{eq:doubleroot}. We recall that the constants $a_k$, $b_k$ are defined in~\eqref{eq:akbk}. The constants $c_k$ introduced in Theorem~\ref{thm:explicitproba} are determined by the recurrence relations
\begin{equation*}
c_{2k}=-c_{2k-1}\frac{k_2^*(a_{2k-1},b_{2k-1})}{k_2^*(a_{2k},b_{2k})}\text{ and }c_{2k+1}=-c_{2k}\frac{k_1^*(a_{2k},b_{2k})}{k_1^*(a_{2k+1},b_{2k+1})}
\end{equation*}
and the normalization relationship $\sum_{k=1}^{2\alpha-1} c_k=1$. 
\end{proposition}

\begin{proof}
For the first relation, let us observe that 
$$\partial_{R^2}\exp(au+bv)|_{v=0}=k_2^*(a,b)\exp(au).$$
We denote $f(u,v)= \mathbb{P}_{(u,v)}(T<\infty)$. Using Theorem \ref{thm:explicitproba}, noticing that $a_{2k-1}=a_{2k}$, we evaluate $\partial_{R^2}f$ at $v=0$ and the Neumann condition~\eqref{eq:neumann} gives
\begin{equation}\label{eq:eq1}
0=\sum_{k=1}^{\alpha-1} \Big(c_{2k-1}k_2^*(a_{2k-1},b_{2k-1})+c_{2k}k_2^*(a_{2k},b_{2k})\Big)\exp(a_{2k}u)+k_2^*(a_{2\alpha-1},b_{2\alpha-1})\exp(a_{2\alpha-1}u).
\end{equation}
By Lemma~\ref{lemma:k*(s)} we see that
$k_2^*(a_{2\alpha-1},b_{2\alpha-1})=k_2^*\circ(\boldsymbol{\mathrm{x}},\boldsymbol{\mathrm{y}})\left(s_2^*\right)=0$
so that the last term in \eqref{eq:eq1} is zero. Under the simple roots condition~\eqref{eq:doubleroot}, $a_{2i}\ne a_{2j}$ for all $i\ne j$, the family $\{u\mapsto \exp(a_{2k}u)\}$ is therefore linearly independent and for all $k$ we obtain
\begin{equation*}
c_{2k-1}k_2^*(a_{2k-1},b_{2k-1})+c_{2k}k_2^*(a_{2k},b_{2k})=0.
\end{equation*}
The proof of the second relation is similar. The normalization comes from the fact that $f(0,0)=1$.

\end{proof}

The following paragraph aims to give a geometric interpretation to all the coefficients $a_k$, $b_k$ and $c_k$ and to explain the compensation phenomenon which appears in the sum of exponentials.

\paragraph{Heuristic of the compensation approach}

Using a recursive compensation method (with a finite number of iterations), it is possible to find a solution to the partial differential equation stated in~\eqref{eq:PDE} and~\eqref{eq:neumann} that is a candidate for being the probability of absorption at the vertex. It is interesting to remark that the positivity of this solution is by no means obvious and that the uniqueness of the solution of this kind of PDE usually requires the positivity of the solution. 

In this paragraph, we explain the compensation phenomenon. By using an analytic approach, we showed in Theorem~\ref{thm:explicitproba} that when $\alpha\in\mathbb{N}$ and~\eqref{eq:doubleroot} holds the absorption probability is
\begin{equation*}
    f(u,v)= \mathbb{P}_{(u,v)}(T<\infty)=\sum_{k=1}^{2\alpha-1} c_k\exp(a_ku+b_kv),
\end{equation*}
where the $(a_k,b_k)$ are determined in Proposition~\ref{rem:recakbk} and the $c_k$ in Proposition~\ref{prop:comp2}.


We define the following function vector spaces 
$$
E_0:=\{ h\in\mathcal{C}^2(\mathbb R_+^2):\;\mathcal{G} h=0 \text{ on }\mathbb R_+^2 \},
$$
$$
E_1:=\{h\in\mathcal{C}^2(\mathbb R_+^2):\;\partial_{R^1}h(0,\,\cdot\,)=0\text{ on }\mathbb R_+\}\text{ and }\; E_2:=\{h\in\mathcal{C}^2(\mathbb R_+^2):\;\partial_{R^2}h(\,\cdot\,,0)=0\text{ on }\mathbb R_+\}.
$$
One may remark that a function $h$ satisfies the PDE~\eqref{eq:PDE} if and only if $h\in E_0$ and $h$ satisfy the Neumann boundary conditions~\eqref{eq:neumann} if and only if $h\in E_1\cap E_2$.
Furthermore, the function $(u,v)\mapsto e^{au+bv}$ belong to $E_0$ if and only if $K(a,b)=0$, belongs to $E_1$ if and only if $k_1^*(a,b)=0$, and belongs to $E_2$ if and only if $k_2^*(a,b)=0$.

By Proposition~\ref{rem:recakbk} all the $(a_k,b_k)$ are on the ellipse $\mathcal{E}$ defined by $K=0$, it is then easy to understand why $f\in E_0$, \textit{i.e.} why $f$ satisfies the partial differential equation~\eqref{eq:PDE}. 

We are now seeking to understand why the coefficients $c_k$ given in Proposition~\ref{prop:comp2} ensure that $f\in E_1\cap E_2$, \textit{i.e.} why $f$ satisfies the Neumann boundary conditions~\eqref{eq:neumann}. 
In fact, the $c_k$ have been chosen such that by grouping the terms of the sum by pairs (except the first or the last term) they \textit{compensate} each other to ensure the inclusions in $E_1$ and $E_2$:
\begin{align*}
f(u,v)&=\underbrace{c_1\exp(a_1u+b_1v)}_{\displaystyle\in E_1} + \sum_{k=1}^{\alpha-1} \underbrace{c_{2k}\exp(a_{2k}u+b_{2k}v)+c_{2k+1}\exp(a_{2k+1}u+b_{2k}v)}_{\displaystyle\in E_1} \\
&=\sum_{k=1}^{\alpha-1} \underbrace{c_{2k-1}\exp(a_{2k}u+b_{2k-1}v)+c_{2k}\exp(a_{2k}u+b_{2k}v)}_{\displaystyle\in E_2} + \underbrace{c_{2\alpha-1}\exp(a_{2\alpha-1}u+b_{2\alpha-1}v)}_{\displaystyle\in E_2}
\end{align*}
so that $f\in E_1\cap E_2$. This is due to the fact that $(c_{2k}k_1^*(a_{2k},b_{2k})+c_{2k+1}k_1^*(a_{2k+1},b_{2k}))e^{b_{2k}v}=0$ and $(c_{2k-1}k_2^*(a_{2k},b_{2k-1})+c_{2k}k_2^*(a_{2k},b_{2k}))e^{a_{2k}u}=0$. 

We now understand the phenomenon of compensation which explains why $f$ is a solution of the partial differential equation~\eqref{eq:PDE} with Neumann boundary conditions~\eqref{eq:neumann}. One may also verify that the limit conditions~\eqref{eq:limits} are also satisfied. Let us note, on the other hand, that the positivity of this function is absolutely not obvious to check.

\paragraph{Double roots.}
This last paragraph deals with the case where $P$ or $Q$ have double roots, \textit{i.e.} when for some integer $j\in\{1,\dots,2\alpha-2\}$, $\theta-2\delta+j\beta\equiv 0\text{ mod}(\pi)$, see Lemma~\ref{lemma:double}.
The number of cases to handle to give a general explicit formula is too big. Nonetheless, we can give the general shape of the absorption probability: if $\alpha\in\mathbb N$, the absorption probability can be written as 
$$\mathbb P_{(u,v)}(T<\infty)=\sum_{k=0}^{2\alpha-1}A_k(u,v)\exp(a_ku+b_kv)$$
where $a_k$ and $b_k$ are given in Equation \eqref{eq:akbk} and the $A_k$ are affine functions of $u$ and $v$. Indeed, Theorem~\ref{thm:laplace} holds even when there are multiple roots. Inverting the Laplace transform we show that the absorption probability can be written 
$\sum_{k=0}^{2\alpha-1}A_k(u,v)\exp(a_ku+b_kv)$
where the $A_k$ are polynomials. A direct calculation shows that $P$ and $Q$ can't have triple roots. This proves that the total degree of $A_k$ is less than $1$ for all $k$. We can also give an intuitive explanation for the fact that there are no triple roots: the geometric interpretation tells us that the sequence $(a_k,b_k)$ cannot visit a point thrice, otherwise it would loop indefinitely.\\

The case where $\alpha=2$ is completely solved below as an example.
\begin{example}[Double roots, $\alpha=2$] 
For $\alpha=2$, we distinguish two cases with double roots 
\begin{itemize}
\item if $\theta - 2 \delta +\beta = - \pi$ then
\begin{equation*}
\mathbb P_{(u,v)}(T<\infty)= \Big(1+c\Big)\exp\Big(\boldsymbol{\mathrm{x}}(s_2)u+\boldsymbol{\mathrm{y}}(s_1)v\Big)-\Big(\boldsymbol{\mathrm{x}}(s_2)u+c\Big)\exp\Big(\boldsymbol{\mathrm{x}}(s_2)u+\boldsymbol{\mathrm{y}}(s_1/\boldsymbol{\mathrm{q}})v\Big)
\end{equation*}
\item if $\theta - 2 \delta +2\beta = - \pi$ then
\begin{equation*}
\mathbb P_{(u,v)}(T<\infty)=-\Big(\boldsymbol{\mathrm{y}}(s_1)v+c\Big)\exp\Big(\boldsymbol{\mathrm{x}}(s_2)u+\boldsymbol{\mathrm{y}}(s_1)v\Big)+\Big(1+c\Big)\exp\Big(\boldsymbol{\mathrm{x}}(s_2\boldsymbol{\mathrm{q}})u+\boldsymbol{\mathrm{y}}(s_1)v\Big)
\end{equation*}
\end{itemize}


where 
$c=\displaystyle\frac{1}{r_1r_2-1}$.
\end{example}

\subsection*{Acknowledgments} 
This project has received funding from Agence Nationale de la Recherche, ANR JCJC programme under the Grant Agreement ANR-22-CE40-0002.
\nb{The authors would like to thank the anonymous referee for his/her helpful comment on the irreducibility of the kernel.}

\bibliographystyle{apalike}



\end{document}